\documentclass [twoside,leqno]{article}

\usepackage{amssymb,amsmath,amstext,amsthm,enumerate}
\usepackage{textcomp} 
\usepackage{graphicx}  

\usepackage{color}

\newcommand{\calF}{\mathcal{F}}

    \def\big#1{{\hbox{$\left#1\vbox to 8.5pt{}\right.\n@space$}}}
    \def\n@space{\nulldelimiterspace=0pt \m@th}
    \def\m@th{\mathsurround=0pt}

    \def\mtx#1#2{\renewcommand{\arraystretch}{1.2}%
        \left( \begin{array}{#1}#2\end{array} \right)}

\def\det{\mathop{\hbox{\rm det}}}
\def\diag{\mathop{\hbox{\rm diag}}}

\def\drop{^{\null}}
\def\half{\textstyle{\frac{1}{2}}}

\def\fourth{\textstyle{\frac{1}{4}}}

\def\inv{^{-1}}
\def\mod#1{|#1|}
\def\norm#1{\Vert{#1}\Vert}
\def\T{^T\!}

\def\nulltag{\gdef\letter{}} \nulltag
\def\range{\mathop{\hbox{\rm range}}}

\def\ftilde{\tilde f}
\def\real{\mathbb{R}} 

\def\Qscr{{\cal Q}}
\def\nthinsp{\mskip -2   mu}
\def\superstar{^{\raise 0.5pt\hbox{$\nthinsp *$}}}

\def\norm#1{\Vert{#1}\Vert}

\def\thrfrths{{\textstyle{\frac{3}{4}}}}
\def\definedas{{\;\equiv\;}}

\def\norm#1{\Vert{#1}\Vert}

\def\base{{\mathbi{b}}}
\def\bb{{\mathbi{b}}}
\def\ba{{\mathbi{a}}}
\def\bc{{\mathbi{c}}}

\def\bA{{\mathbi{A}}}
\def\bB{{\mathbi{B}}}
\def\bC{{\mathbi{C}}}

\def\bg{{\mathbi{g}}}

\def\bp{{\mathbi{p}}}
\def\bptrans{\trans{\mathbi{p}}}

\def\bm{{\mathbi{m}}}

\def\boldp{{\bf p}}

\def\boldw{{\bf w}}

\def\btriangell{{\mathbi{\triang}_{\ell}}}
\def\btriangzer{{\mathbi{\triang}_{0}}}
\def\btriangone{{\mathbi{\triang}_{1}}}

\def\bP{{\mathbi{P}}}

\def\bv{{\mathbi{v}}}
\def\bvtrans{\trans{\mathbi{v}}}
\def\boldv{{\bf v}}

\def\expco{\chi}

\def\sublow{{_{\rm low}}}
\def\subhigh{{_{\rm high}}}
\def\subymid{{_{\rm mid}}}
\def\subxmid{{_{\rm mid}}}

\def\subbest{{_{\rm best}}}
\def\subworst{{_{\rm worst}}}
\def\subnext{{_{\rm next}}}
\def\subref{{_{\rm r}}}
\def\subout{{_{\rm out}}}
\def\subin{{_{\rm in}}}
\def\subleft{{_{\rm left}}}
\def\subright{{_{\rm right}}}
\def\flatconst{{\rho}}

\def\remain{r}
\def\objfun{f}
\def\fivehalves{\textstyle{\frac52}}
\def\threehalves{\textstyle{\frac32}}
\def\xtrsubrsq{{\xtrans_{\rm r}^2}}
\def\trans{\widetilde}
\def\Atrans{{\trans A}}

\def\htrans{{\trans h}}
\def\ytrans{{\trans y}}
\def\xtrans{{\trans x}}
\def\superlimit{^{\dag}}

\def\wtrans{{\trans w}}
\def\bptrans{{\trans\bp}}

\def\bplimit{{\bp\superlimit}}
\def\bptranslimit{{\bptrans\superlimit}} 
\def\definedas{:=}
\def\fgrad{\bg}
\def\fgradhat{{\widehat\fgrad}}

\def\threefourths{{\textstyle{\frac{3}{4}}}}
\def\fivefourths{{\textstyle{\frac{5}{4}}}}

\def\movedist{\kappa}

\def\thrfrths{{\textstyle{\frac{3}{4}}}}

\def\deltax{{\delta_x}}
\def\deltay{{\delta_y}}
\def\triang{{\Delta}}
\def\triangtrans{{\trans\triang}}

\def\flatness{{\Gamma}}

\def\mathbi#1{\boldsymbol{#1}}
\newcommand{\coordframe}{{\mathfrak F}}

\numberwithin{equation}{section} 

\usepackage{amsthm}

\newtheoremstyle{theorem-slant}
   {\topsep}
   {\topsep}
   {\slshape}
   {}
   {\bfseries}
   {}
   {.5em}
   {\thmname{#1}\thmnumber{ #2}.\thmnote{ (#3)}}

\theoremstyle{theorem-slant}
\newtheorem{theorem}{Theorem}[section]
\newtheorem{lemma}[theorem]{Lemma}

\newtheorem{proposition}[theorem]{Proposition}
\newtheorem{hypothesis}{Hypothesis}

\newtheoremstyle{definition-new}
   {\topsep}
   {\topsep}
   {}
   {}
   {\bfseries}
   {}
   {.5em}
   {\thmname{#1}\thmnumber{ #2}.\thmnote{ (#3)}}

\theoremstyle{definition-new}

\newtheorem{definition}[theorem]{Definition}

\newtheorem{example}[theorem]{Example}

\theoremstyle{remark}
\newtheorem{remark}[theorem]{Remark}

\def\abstract{\if@twocolumn
\section*{Abstract}
\else
\begin{center}
{\bf Abstract}
\end{center}
\quotation
\fi}
\def\endabstract{\if@twocolumn\else\endquotation\fi}

\title{Convergence of the restricted Nelder--Mead algorithm\\
in two dimensions}

\author{Jeffrey C. Lagarias\thanks{Department of Mathematics,
University of Michigan, Ann Arbor, Michigan 48109
(lagarias$@$umich.edu). The work of this author is partially supported by
National Science Foundation grant DMS-0801029.
The author also received support through the Mathematics Research Center
at Stanford University.}\\
   Bjorn Poonen\thanks{Department of Mathematics,
Massachusetts Institute of Technology, Cambridge, Massachusetts 02139
(poonen$@$math.mit.edu).  The work of this author is
supported by National Science Foundation grant DMS-0841321.}\\
   Margaret H. Wright\thanks{Computer Science Department,
New York University, New York, New York 10012 (mhw$@$cs.nyu.edu).
The work of this author is partially supported by
Department of Energy grant DE-FG02-88ER25053.}
}

\date{April 2, 2011}

\begin{document}

   \input epsf

\maketitle \thispagestyle{empty}

\begin{abstract}
The Nelder--Mead algorithm, a longstanding
direct search method for unconstrained
optimization published in 1965, is designed to minimize
a scalar-valued function $f$ of $n$ real variables
using only function values, without any derivative information.
Each Nelder--Mead iteration is associated with a
nondegenerate simplex defined by $n + 1$ vertices and
their function values; a typical iteration produces a new simplex
by replacing the worst vertex by a new point.  
Despite the method's widespread use,
theoretical results have been limited:
for strictly convex objective functions of one variable
with bounded level sets,
the algorithm always converges to the minimizer;
for such functions of two variables,
the diameter of the simplex converges to zero,
but examples constructed by McKinnon show
that the algorithm may converge to a nonminimizing point.

This paper considers the {\sl restricted\/} Nelder--Mead algorithm,
a variant that does not allow expansion steps.
In two dimensions we show that, for any nondegenerate starting simplex and
any twice-continuously differentiable function
with positive definite Hessian and bounded level sets,
the algorithm always converges to the minimizer.
The proof is based on treating the method as a discrete dynamical system,
and relies on several techniques that are non-standard in convergence
proofs for unconstrained optimization.
\end{abstract}

\abovedisplayskip 5pt plus2pt minus4pt
\belowdisplayskip \abovedisplayskip
\abovedisplayshortskip 3 pt plus 3pt
\belowdisplayshortskip 3pt plus3pt minus3pt
\setcounter{page}{1}  %

%
\section{Introduction} \label{sec-introduction}

Since the mid-1980s, interest has steadily grown in
{\sl derivative-free\/} methods (also called {\sl non-derivative\/} methods)
for solving optimization problems,
unconstrained and constrained.
Derivative-free methods that adaptively
construct a local model of relevant nonlinear functions
are often described as ``model-based'',
and derivative-free methods that do not explicitly involve such a model
tend to be called ``direct search'' methods.   
See
\cite{CSV09} for a recent survey of derivative-free methods;
discussions focusing on direct search methods include, 
for example, \cite{Dundee,Kelley-book,LTT00,KLT-SIREV,PowellActa}.

The Nelder--Mead (NM) simplex method \cite{NM} is a direct search method. 
Each iteration of the NM method begins with a
nondegenerate simplex (a geometric figure in $n$ dimensions
of nonzero volume that is the convex hull of $n+1$ vertices),
defined by its vertices and the associated values of $f$.
One or more trial points are computed, along with their
function values, and the iteration produces
a new (different) simplex such that the function values at
its vertices typically satisfy a descent condition
compared to the previous simplex.

The NM method is appealingly 
simple to describe (see Figure~\ref{fig-nmmoves}),
and has been widely used (along with numerous variants)
for more than 45 years, in many scientific 
and engineering applications. 
But little mathematical
analysis of any kind of the method's performance has appeared,
with a few exceptions such as \cite{WoodsPhD,HSW88} (from more than
20 years ago) and (more recently) \cite{HanNeumann06}.
As we discuss in more detail below,
obtaining even limited convergence proofs for the original method
has turned out to be far from simple.
The shortage of theory, plus the discovery of low-dimensional 
counterexamples (see~\eqref{eqn-mckinsamp})
have made the NM method an outlier among
modern direct search methods, which are deliberately
based on a rigorous mathematical foundation.  (See, for
example, \cite{CoopePrice01,AD03,KLT-SIREV,Audet04},
as well as more recent publications about direct search
methods for constrained problems.) 
Nevertheless the NM method retains importance
because of its continued use and availability in computer packages (see
\cite{Recipes,matlab,gsl})
and its apparent usefulness in some situations.

In an effort to develop positive theory about the original
NM algorithm, an analysis of its convergence behavior
was initiated in \cite{LRWW98} in 1998, along with resolution of
ambiguities in \cite{NM} about whether function comparisons involve
``greater than'' or ``greater than or equal''
tests.\footnote{Resolution of these ambiguities can have a
noticeable effect on the
performance of the algorithm; see \cite{Gurson00}.}
In what follows we use the term {\sl Nelder-Mead algorithm\/} to
refer generically to one of the precisely specified procedures in \cite{LRWW98};
these contain a  number of adjustable parameters (coefficients),
and the {\sl standard coefficients\/} represent an often-used choice.
For strictly convex objective functions with bounded level
sets, \cite{LRWW98} showed
convergence of the most general form of the NM algorithm to the
minimizer in one dimension.  For the NM algorithm
with standard coefficients in dimension two, where the simplex
is a triangle, it was shown that
the function values at the simplex vertices converge to
a limiting value, and furthermore that the diameter of the simplices
converges to zero.   But it was not shown that
the simplices always converge to a limiting point, and
up to now  this question
remains unresolved.

Taking the opposite perspective,
McKinnon~\cite{McKinnon98} devised a family of
two-dimensional counterexamples consisting of strictly convex
functions with bounded level sets and a specified initial
simplex, for which the NM simplices
converge to a nonminimizing point.
In the smoothest McKinnon example, the objective function is
\begin{equation}
\label{eqn-mckinsamp}
f_m(x, y) = \left\{ \begin{array}{ll}
             2400 |x|^3 + y + y^2 &~\mbox{if} ~~x \le 0\\
             6 x^3 + y + y^2 & ~\mbox{if}~~ x \ge 0,
\end{array}
\right.
\end{equation}
when the vertices of the starting simplex are
$(0,0)$, $(1,1)$ and $((1+\sqrt{33})/8,\; (1-\sqrt{33})/8))$.
Note that $f_m$ is twice-continuously differentiable
and that its Hessian is positive definite except at the
origin, where it is singular.  As
shown in Figure~\ref{fig-mckinfig}, the NM algorithm
converges to the origin (one of the initial vertices)
rather than to the minimizer $(0, -\frac{1}{2})$,
performing an infinite sequence of inside contractions
(see Section~\ref{sec-nmalgs}) in which the best vertex of
the initial triangle is never replaced. 
%
%
\begin{figure}[htb]
\epsfxsize = 4in
\epsfysize = 2.25in
$$\epsfbox{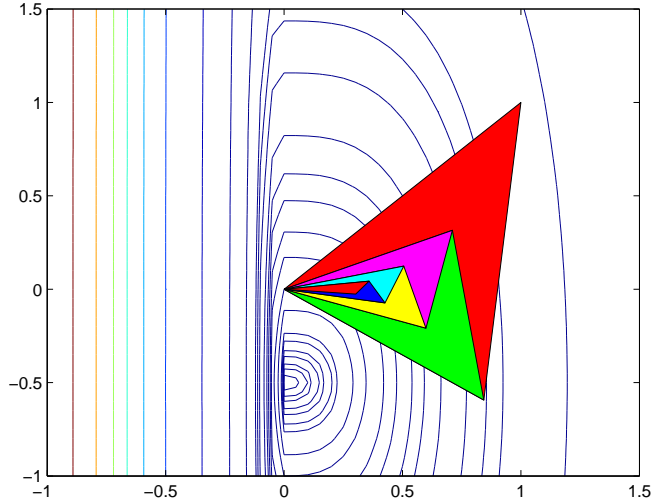}$$
\caption{\small The NM algorithm's failure on the
McKinnon counterexample~\eqref{eqn-mckinsamp}.}
\label{fig-mckinfig}
\end{figure}

Functions proposed by various authors on which the NM algorithm
fails to converge to a minimizer are surveyed in~\cite{McKinnon98},
but counterexamples in the McKinnon family illustrated by 
\eqref{eqn-mckinsamp} constitute the ``nicest'' functions for which
the NM algorithm converges to a non-stationary point.

An algorithmic flaw that has been observed is that the iterations ``stagnate''
or ``stall'', often because the simplex becomes
increasingly close to degenerate (as depicted in Figure~\ref{fig-mckinfig}).  
Previously proposed corrective strategies include: placing
more restrictions on moves that decrease
the size of the simplex; imposing a
``sufficient decrease'' condition (stronger than simple
decrease) for accepting a new vertex; and
resetting the simplex to one that is ``nice''.
See, for example,
\cite{Rykov83,WoodsPhD,Tse99,Kel99,Priceetal02,NT02,CSV09}, a
small selection of the many papers that
include convergence results for modifications of Nelder--Mead.

Our object in this paper is to fill in additional theory for the NM
algorithm in the two-dimensional case, which remains of
interest in its own right.
As noted by McKinnon \cite[page 148]{McKinnon98}, it is not even known
whether
the NM algorithm converges for the prototypically nice 
function $\objfun(x,y)= x^2+y^2$.
Here we answer this question
affirmatively for a simplified variant of the NM algorithm,
where the simplification reduces the number of allowable moves
rather than attempting to ``fix'' the method.
In the original NM algorithm
(see Section~\ref{sec-nmalgs}), the allowable moves
are reflection, expansion, outside contraction,
inside contraction, and shrink; 
an expansion doubles the volume of an NM simplex, while
all other moves either leave the volume the same or decrease it.
An expansion is tried only after the
reflection point produces a strict improvement in the
best value of $\objfun$; the motivation is to allow
a longer step along an apparently promising direction.
The {\sl restricted\/} Nelder--Mead (RNM) algorithm defined
in Section~\ref{sec-nmalgs} does not allow expansion steps.
Thus we are in effect considering a ``small step'' NM algorithm.

Our analysis applies to the following class of functions:
\begin{definition}
Let $\calF$ denote the class of twice-continuously
differentiable functions $\objfun \colon \real^2 \to \real$ 
with bounded level sets and everywhere positive definite Hessian.
\end{definition}
The class $\calF$ is a
subclass of those considered in \cite{LRWW98}, where there is
no requirement of differentiability.

The contribution of this paper is to prove
convergence of the restricted Nelder-Mead algorithm for functions
in $\calF$:

%
%

\begin{theorem}[appears again as Theorem~\ref{thm-converges}]
\label{thm-rnmconv}
If the RNM algorithm is applied to 
a function $\objfun \in \calF$,
starting from any nondegenerate triangle, 
then the algorithm converges to the unique minimizer of $\objfun$.
\end{theorem}

\begin{remark}
Theorem~\ref{thm-rnmconv}
immediately implies a generalization to a larger class of functions.
Namely, if $\objfun \in \calF$, 
and $g\colon \real \rightarrow \real$ is a strictly increasing function,
then the RNM algorithm applied to $\ftilde := g \circ \objfun$ converges,
because the RNM steps for $\ftilde$ 
are identical to those for $\objfun$.
\end{remark}

\begin{remark}
Because the NM iterations 
in the McKinnon examples include no expansion steps, the
RNM algorithm also will fail
to converge to a minimizer on these examples.
It follows that, in order to obtain a positive convergence result,
additional assumptions on the function over those in \cite{LRWW98}
must be imposed.
In particular, the positive-definiteness condition on the Hessian
in Theorem~\ref{thm-rnmconv} rules out the smoothest
McKinnon example (\ref{eqn-mckinsamp}), in which the Hessian is singular at
the origin (the nonminimizing initial vertex to which the
NM algorithm converges).
\end{remark}

An interesting general property of the Nelder--Mead algorithm
is the constantly changing shape of the simplex
as the algorithm progresses.
Understanding the varying geometry of the simplex seems
crucial to explaining how the algorithm behaves.
Our proof of Theorem~\ref{thm-rnmconv}
analyzes the RNM algorithm as a discrete dynamical
system, in which the shapes of the relevant simplices
(with a proper scaling) form 
a phase-space for the algorithm's behavior. The imposed
hypothesis on the Hessian, which is stronger than strict convexity,
allows a crucial connection to be made between a (rescaled) local
geometry and the vertex function values.
We analyze the algorithm's behavior in
a transformed coordinate system that corrects for this rescaling.

The proof of Theorem~\ref{thm-rnmconv} establishes convergence by
contradiction, by showing that the algorithm can find no way not
to converge. We make, in effect, a ``Sherlock Holmes'' argument: Once you
have eliminated the impossible, whatever remains, however improbable,
must be the truth.\footnote{A.\ Conan Doyle, ``The Sign of the Four'',
{\sl Lippincott's Monthly Magazine}, February 1890.}
We show that, in order not to converge to the minimizer, the triangles
would need to flatten
out according to a particular geometric scaling, but there
is no set of RNM steps permitting this
flattening to happen.  This result is confirmed
through an auxiliary potential function measuring
the deviation from scaling.
One can almost say that the RNM algorithm converges
in spite of itself.

%
%

\section{The restricted Nelder--Mead algorithm}\label{sec-nmalgs}

Let $f \colon \real^n \to \real$ be a function to be minimized,
and let $\boldp_1,\ldots,\boldp_{n+1}$ be the vertices of
a nondegenerate simplex in $\real^n$.
One iteration of the RNM algorithm (with standard coefficients)
replaces the simplex by a new one according to the following
procedure.

\smallskip
\noindent
{\bf One iteration of the standard RNM algorithm.}
\begin{enumerate}
\item {\bf Order.}
Order and label the $n+1$ vertices to satisfy
$f(\boldp_1) \le f(\boldp_2) \leq \cdots \le f(\boldp_{n+1})$,
using appropriate tie-breaking rules such as those in \cite{LRWW98}.

\item \label{step-reflect}
{\bf Reflect.}
Calculate $\bar\boldp = \sum_{i=1}^n \boldp_i/n$, the average of the $n$
best points (omitting $\boldp_{n+1}$).
Compute the {\sl reflection point\/} $\boldp\subref$, defined as
$\boldp\subref =  2\bar\boldp - \boldp_{n+1}$,
and evaluate $f\subref = f(\boldp\subref)$.
If $f\subref < f_n$, accept the reflected point $\boldp\subref$
and terminate
the iteration.

\item\label{step-contract}
{\bf Contract.}
If $f\subref \ge f_n$,
perform a {\sl contraction} between $\bar\boldp$ and the better of
$\boldp_{n+1}$ and $\boldp\subref$.

{\bf a.  Outside contract.}\quad
If $f_n \le f\subref < f_{n+1}$ (i.e., $\boldp\subref$ is strictly
better than
$\boldp_{n+1}$), perform an {\sl outside contraction\/}:
calculate the outside contraction point
$\boldp\subout = \half(\bar\boldp +\boldp\subref)$,
and evaluate $f\subout = f(\boldp\subout)$.  If $f\subout \le f\subref$, accept
$\boldp\subout$ and terminate the iteration;
otherwise, go to Step \ref{step-shrink}
(perform a shrink).

{\bf b.  Inside contract.}\quad
If $f\subref \ge f_{n+1}$, perform an {\sl inside contraction\/}:
calculate the inside contraction point
$\boldp\subin = \half(\bar\boldp + \boldp_{n+1})$,
and evaluate
$f\subin = f(\boldp\subin)$.  If $f\subin < f_{n+1}$, accept
$\boldp\subin$ and terminate the iteration; otherwise, go to
Step \ref{step-shrink}
(perform a shrink).

\item\label{step-shrink}
{\bf Perform a shrink step.}
Evaluate $f$ at the $n$ points
$\bv_i = \half(\boldp_1 + \boldp_i)$, $i = 2$, \dots, $n+1$.
The (unordered) vertices of the simplex at the
next iteration consist of $\boldp_1$, $\bv_2$, \dots, $\bv_{n+1}$.
\end{enumerate}
The result of an RNM iteration is either: (1) a single new
vertex---the {\sl accepted point}---that
replaces the worst vertex $\boldp_{n+1}$ in the set of vertices for the next
iteration; or (2) if a shrink is
performed, a set of $n$ new points that, together with
$\boldp_1$, form the simplex at the next iteration.

%
%

\begin{figure}[htb]
$$\epsfbox{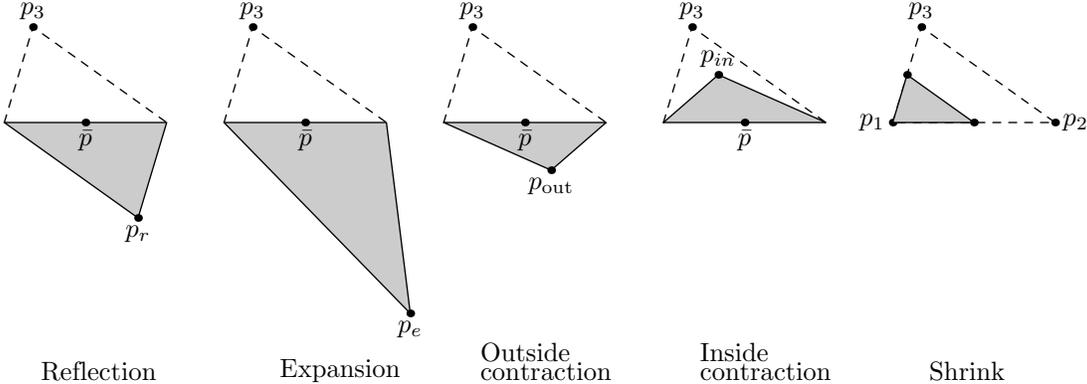}$$
\caption{\small The five possible moves in the original NM algorithm
are shown.  The original simplex is surrounded by a
dashed line, and its worst vertex is labeled $\boldp_3$.
The point $\bar\boldp$ is the average of the two best vertices.
The shaded figures are NM simplices following reflection,
expansion, outside contraction, inside contraction, and shrink,
respectively.
(In the ``shrink'' figure, the best vertex is labeled $\boldp_1$.)
The ``expansion'' step is omitted in the RNM algorithm.}
\label{fig-nmmoves}
\end{figure}

Starting from a given nondegenerate simplex,
let $\boldp_1^{(k)}$, \dots, $\boldp_{n+1}^{(k)}$
be the vertices at the {\sl start} of the $k^{\textup{th}}$ iteration.
Let $\mathbf{z} \in \real^n$ be a point.
We say that the RNM algorithm converges to $\mathbf{z}$ if
$\lim_{k \to \infty} \boldp_i^{(k)} = \mathbf{z}$ 
for every $i \in \{1,\ldots,n+1\}$.

%
%
\begin{remark}
In two dimensions, a reflect step performs a $180^\circ$ 
rotation of the triangle around $\bar\boldp$, so the resulting
triangle is congruent to the original one.
But in higher dimensions, the reflected simplex is not congruent
to the original.
\end{remark}

%
%
\begin{remark}
Shrink steps are irrelevant in this paper because we are concerned only
with strictly convex objective functions, for which shrinks cannot
occur (Lemma 3.5 of \cite{LRWW98}).  It follows that, at
each NM iteration, the function value at the new vertex is
strictly less
than the worst function value at the previous iteration.
\end{remark}

%
%
\begin{remark}
The original Nelder--Mead algorithm differs from the above
in Step~2.  Namely, if $\boldp\subref$ 
is better than all $n+1$ of the vertices,
the original NM algorithm tries evaluating $f$ at the
{\sl expansion\/} point
$\boldp_{\rm e} \definedas \bar\boldp + \expco (\bar\boldp - \boldp_{n+1})$
for a fixed expansion coefficient $\expco>1$, and
the worst vertex $\boldp_{n+1}$ is then replaced by the
better of $\boldp_{\rm e}$ and $\boldp\subref$.
In fact, Nelder and Mead proposed a family of algorithms,
depending on coefficients for reflection, contraction, and shrinkage
in addition to expansion.
A complete, precise definition of an NM iteration
is given in \cite{LRWW98}, along with a set of tie-breaking rules.
Instances of the moves in the original NM algorithm are shown in
Figure~\ref{fig-nmmoves}.
\end{remark}

%
%
\begin{remark}
One feature of the RNM algorithm 
that makes it easier to analyze than the original algorithm
is that the volume of the simplex is non-increasing at each step.
The volume thus serves as a 
Lyapunov function.\footnote{See Definition~1.3.4 in \cite[page 23]{StuHumph}.}
\end{remark}

We henceforth consider the RNM algorithm 
{\sl in dimension two}, for which it is known that the
simplex diameter converges to zero.

%
%

\begin{lemma}\label{lem-diamtozero}
Suppose that the RNM algorithm is applied to a
strictly convex $2$-variable function 
with bounded level sets. Then for any
nondegenerate initial simplex, the diameters of the RNM simplices
(triangles) produced by the algorithm converge to $0$.
\end{lemma}
\begin{proof}
The proof given in \cite[Lemma~5.2]{LRWW98} for the original
NM algorithm applies even when expansion steps are disabled.
\end{proof}

%
%
\section{Convergence}\label{sec-convproof}

%
%

\subsection{The big picture}
\label{sec-bigpic}

Because the logic of the
convergence proof is complicated, 
we begin with
an overview of the argument.
Each $\objfun \in \calF$ is strictly convex, so by Lemma~\ref{lem-diamtozero}
we know that 
 the evolution of any triangle under the RNM algorithm 
has the diameter of the triangle
converging to zero.
(We do not yet know that the triangles converge
to a limit point.)  
The convergence proof proceeds by contradiction, making an initial hypothesis
(Hypothesis~1 in Section~\ref{sec-levelset}) that the (unique) minimizer
of $\objfun$ is not a limit point of the RNM triangles.  Under
this condition,
all three RNM vertices must approach a level set corresponding
to a function value strictly higher than the optimal value.  
By our assumptions on $\objfun$,
this level set is a strictly convex closed
curve with a continuously differentiable tangent vector.

The RNM triangle must become small as it approaches this bounding
level set.  Therefore, from the viewpoint of the triangle, blown up
to have (say) unit diameter, the level set flattens out to
a straight line.  The heuristic underlying our argument is that, in
order for this to happen, the triangle must itself have its shape
flatten out, with its width in the level set direction
(nearly horizontal, as seen from the triangle) being roughly the
square root of its height in the perpendicular direction. 
In particular, its width becomes proportionally much larger than its height.
A local coordinate frame (Section~\ref{sec-coordframe}) is defined 
in order to describe this phenomenon.

At the {\sl start\/} of iteration $k$,
we measure area and width in a local coordinate frame,
and define a quantity called ``flatness'' by 
$\flatness_k \definedas \hbox{\rm area}_k\drop/\hbox{\rm width}_k^3$.
If a reflection is taken during iteration $k$ and the same
coordinate frame is retained,
the area and width of the RNM triangle at iteration $k+1$ remain
the same, so $\flatness_{k+1} = \flatness_k$.
Hence, in order for the diameter to
converge to zero (Lemma~\ref{lem-diamtozero}), 
there must be infinitely many contraction steps.
We show that, at a sufficiently advanced iteration $k$ of the
RNM algorithm, a necessary condition for a contraction to occur
is that $\flatness_k \le 10$; we also show that
the value of $\flatness$ eventually
unavoidably increases as the algorithm proceeds.
A contradiction thus arises because no combination of the permitted
reflection and contraction steps allows the needed square root
rate of decrease. 

The argument is complicated because the local coordinate frame changes at
every step.  Near the end of the proof (in Proposition~\ref{prop-14steps}),
we analyze sequences of no more than 14 steps, beginning
with a contraction, in an
advanced phase of the algorithm.  Using a
coordinate frame defined by a vertex of the first triangle
in the sequence, we show that switching to a new coordinate
system defined via the final triangle in the sequence makes only
a small change in the flatness.  This allows us to show that the
flatness is inflated by a factor of at least $1.01$ 
after at most $14$ steps, which eventually means that a contraction
cannot be taken.  Since the triangle cannot reflect forever,
our contradiction
hypothesis must have been false; i.e., the method must converge.

%
%

\subsection{Notation}

Points in two dimensions are denoted by boldface lower-case letters,
but a generic point is often called $\bp$, which is treated as
a column vector and written as $\bp = (x, y)^T$.
We shall also often use an affinely transformed coordinate system
with generic point denoted by $\bptrans= (\xtrans, \ytrans)^T$.
To stress the $(x,y)$ coordinates of a specific
point, say $\bb$, we write $\bb = (b_x, b_y)^T$.

For future reference, we explicitly give the formulas for the
reflection and contraction points in two dimensions:
\begin{eqnarray}
\bp\subref & = &
\bp_1 + \bp_2 - \bp_3\qquad\qquad \hbox{(2-d reflection)};
\label{eqn-twodref}\\
\bp\subout & = & \thrfrths(\bp_1 + \bp_2) - \half\bp_3
\qquad\hbox{(2-d outside contraction)};
\label{eqn-twodoutc}\\
\bp\subin & = & \fourth(\bp_1 + \bp_2) + \half\bp_3
\qquad\hbox{(2-d inside contraction)}.
\label{eqn-twodinc}
\end{eqnarray}
Given the three vertices of a triangle,
the reflection and contraction points
depend only on which (one) vertex is labeled as ``worst''.

%
%

\subsection{A changing local coordinate system}
\label{sec-coordframe}

The type of move at each RNM iteration
is governed by a discrete decision,
based on comparing values of $f$.
Heuristically, for a very small triangle near a point $\base$,
the result of the comparison is usually unchanged
if we replace $f$ by its degree-$2$ Taylor polynomial
centered at $\base$.
If $\base$ is a nonminimizing point,
then we can simplify the function further by making
an affine transformation into a new coordinate system
$\bptrans=(\xtrans,\ytrans)$ (depending on $\base$)
in which the Taylor polynomial has the form
\[
	\textup{constant} + \ytrans + \half \xtrans^2.
\]
This motivates the following lemma, which is a version of Taylor's theorem.

%
%

\begin{lemma}[Definition of local coordinate frame.]
\label{lem-affdef}
Let $\objfun \in \calF$.
Given a point $\base$ and a nonsingular $2 \times 2$ matrix $M$,
we may define an affine transformation
\begin{equation}
\label{eqn-affinet}
    \bptrans = M\inv(\bp-\base)
\end{equation}
(with inverse map $\bp = M\bptrans + \base$).
\begin{enumerate}[\upshape (i)]
\item\label{I:existence of M}
For each point $\base$ that is not the minimizer of $\objfun$,
there exists a unique $M$ with $\det M>0$
such that when
the function $\objfun$ of $\bp = (x,y)^T$ is re-expressed in
the new coordinate system $\bptrans=(\xtrans,\ytrans)^T$ above,
the result has the form
\begin{equation}
\label{eqn-objtransform}
    \objfun(\bp)  = \objfun(\bb) + \ytrans + \half \xtrans^2 + \remain(\xtrans,\ytrans),
\end{equation}
where $\remain$ is an error term satisfying
\begin{equation}
\label{eqn-remaindef}
   \remain(\xtrans,\ytrans) = \half\alpha\ytrans^2 +
     o(\max(|\xtrans|^2, |\ytrans|^2)),
\end{equation}
as $(\xtrans,\ytrans) \to \mathbf{0}$ (i.e., as $(x,y)^T \to \base$),
for some $\alpha>0$.
\item\label{I:function r}
The function $\remain$ in~\eqref{I:existence of M} satisfies
$d\remain/d\xtrans = o(\max(|\xtrans|,|\ytrans|))$ and
$d\remain/d\ytrans = o(|\xtrans|) + O(|\ytrans|)$,
and the rate at which the $o(\cdot)$ terms approach zero and
the bounds implied by $O(\cdot)$ can be made uniform for
$\base$ in any compact set 
not containing the minimizer of $\objfun$.

\item\label{I:uniformly continuous} 
As $\base$ varies over a compact set
not containing the minimizer of $\objfun$, the matrices $M$ and $M^{-1}$
are bounded in norm and uniformly continuous.

\end{enumerate}
\end{lemma}

\begin{proof}
Let $\fgrad = \nabla f(\base)$ and $H = \nabla^2 f(\base)$
denote, respectively, the gradient and Hessian matrix of
$\objfun$ at $\base$.  
Since $\objfun$ is strictly convex, its gradient can vanish
only at the unique minimizer, so $\fgrad \ne 0$.
Because $f$ is twice-continuously
differentiable, we can expand it in Taylor series
around $\base$:
\begin{align}
\label{eqn-taylorexp}
\objfun(\bp) &= \objfun(\base)  +
     \fgrad^T (\bp-\base) +
     \half (\bp-\base)^T H (\bp-\base) + o(\norm{\bp-\base}^2) \\
\label{eqn-taylorexp-trans}
	&= \objfun(\base)  +
    \fgrad^T M \bptrans +
      \half \bptrans^T M^T H M \bptrans + o(\norm{\bp-\base}^2).
\end{align}
The Taylor expansion \eqref{eqn-taylorexp-trans} has the desired form if
\[
    \fgrad\T M   = (0\;\;1)\quad\hbox{and}\quad
      M\T H M  =  \mtx{cc}{1 & 0 \\
                                  0 & \alpha}.
\]
for some $\alpha>0$.
In terms of the columns $\bm_1$ and $\bm_2$ of $M$, these conditions say
\[
	\fgrad^T \bm_1 = 0, \quad 
	\fgrad^T \bm_2 = 1, \quad 
	\bm_1^T H \bm_1 = 1, \quad 
	\bm_1^T H \bm_2 = 0,
\]
and then we may set $\alpha \definedas \bm_2^T H \bm_2$,
which will be positive since $H$ is positive definite
and since the conditions above force $\bm_2$ to be nonzero.

Since $\fgrad \ne 0$,
the condition $\fgrad^T \bm_1=0$ says that $\bm_1$ is a multiple of
the vector $\fgradhat$ obtained by rotating $\fgrad$ by $90^\circ$ clockwise:
$\bm_1=\xi_1 \fgradhat$ for some $\xi_1$.
The condition $\bm_1^T H \bm_1 = 1$ implies that $\bm_1 \ne \mathbf{0}$.
The condition $\bm_1^T H \bm_2 = 0$ says that $H\bm_2$ is a multiple
of $\fgrad$.
Since $H$ is positive definite, $H$ is nonsingular,
so the equation $H \boldw = \fgrad$ 
has the unique solution $\boldw = H^{-1} \fgrad$, 
and then $\bm_2 =\xi_2 \boldw$ for some $\xi_2$.
The normalizations $\fgrad\T \bm_2 = 1$ and $\bm_1^T H \bm_1\drop = 1$
are equivalent to
\begin{equation}
    \xi_2 = \frac{1}{\fgrad^T \boldw} = \frac{1}{\boldw^T H \boldw}
	\quad\hbox{and}\quad
    \xi_1^2 = \frac{1}{\fgradhat^T H \fgradhat};
\label{eqn-xidef}
\end{equation}
the denominators are positive since $H$ is positive definite
and $\boldw$ and $\fgradhat$ are nonzero.
These conditions
determine $M$ uniquely up to the choice of sign of its first column,
i.e., the sign of $\xi_1$,
but we have not yet imposed the condition $\det M>0$.
We claim that it is the positive choice of $\xi_1$ that makes $\det M>0$:
since $\bm_1$ and $\bm_2$ are then positive multiples of $\fgradhat$
and $\boldw$, respectively, the condition $\det M>0$ is equivalent to
$\fgrad^T \boldw > 0$, or equivalently, $\boldw^T H^T \boldw>0$,
which is true since the matrix $H^T=H$ is positive definite.
This proves~\eqref{I:existence of M}.

Since $\objfun$ is twice-continuously differentiable,
$\fgrad$ and $H$ vary continuously as $\base$ varies within
a compact set not containing the minimizer of $\objfun$.
Hence $M$ and $M^{-1}$ vary continuously as well.
This proves \eqref{I:function r} and \eqref{I:uniformly continuous}.
\end{proof}

%
%
\begin{remark}
If $H$ is positive semidefinite and singular,
then the equation $H\boldw=\fgrad$ continues
to have a solution provided that $\fgrad \in \range(H)$,
as in the McKinnon example~\eqref{eqn-mckinsamp}.  
But in this case, $H\fgradhat = 0$,
so $H\bm_1 = 0$, which contradicts $\bm_1^T H \bm_1\drop = 1$,
and no matrix $M$ exists.
\end{remark}

%
%
\begin{remark}
As $\base$ approaches the minimizer of $\objfun$,
we have $\fgrad \to \mathbf{0}$,
and the formulas obtained in the proof of Lemma~\ref{lem-affdef}
show that $\bm_1$ remains bounded 
while $\bm_2$ and the value of $\alpha$ ``blow up'', 
so $M$ becomes unbounded in norm with an increasing condition number.
\end{remark}

The local coordinate frame defined in Lemma~\ref{lem-affdef}
depends on the base point $\base$, the gradient vector $\fgrad$,
and the Hessian matrix $H$.
In the rest of this section, we use
${\coordframe}(\base)$ (with a nonminimizing point $\base$ as argument)
to denote the local coordinate frame with base point $\base$. In
the context of a sequence of RNM iterations,
$\coordframe_k$ (or $\coordframe(\triang_k)$, with a subscripted
RNM triangle as argument) will mean the coordinate
frame defined with a specified base point in RNM triangle
$\triang_k$.

%
%

\subsubsection{Width, height, area, and flatness.}
This section collects some results about transformed
RNM triangles.

%
%

\begin{definition}[Width, height, and flatness.]
\label{def-flatness}
Let $\objfun \in \calF$, and
let $\triang$ denote a nondegenerate triangle that lies in a
compact set $\Qscr$ not containing
the minimizer of $\objfun$.  Assume that we are given
a base point $\base$ in $\Qscr$, along with the
coordinate frame defined at $\base$ as in Lemma~{\rm\ref{lem-affdef}}.
\begin{itemize}
\item The (transformed) {\sl width\/} of $\triang$, denoted by
$\wtrans(\triang)$, is the maximum absolute value of the difference in
$\xtrans$-coordinates of two vertices of $\triang$;
\item
The (transformed) {\sl height\/}, denoted by
$\htrans(\triang)$,
is the maximum absolute value of the
difference of $\ytrans$-coordinates of two vertices of $\triang$;
\item The {\sl flatness\/} of $\triang$, denoted by
$\flatness(\triang)$, is
\begin{equation}
     \flatness(\triang) \definedas \frac{\Atrans(\triang)}{\wtrans(\triang)^3},
\label{eqn-flatnessdef}
\end{equation}
where $\Atrans(\triang)$ is the (positive)
area of $\triang$ measured in the transformed coordinates.
\end{itemize}
The argument $\triang$ may be omitted when it is obvious.
\end{definition}

%
%

\begin{lemma}[Effects of a reflection]
\label{lem-refeffects}
The (transformed) height and width of an RNM triangle
are the same as those of its reflection,
if the same base point is used to define the local coordinate frame
for both triangles.
\end{lemma}

\begin{proof}
The new triangle is a $180^\circ$ rotation of the old triangle.
\end{proof}

The next lemma bounds the change in three quantities
arising from small changes in the base point used
for the local coordinate frames.
In~\eqref{I:closebase flatness},
we need a hypothesis on the width and height
since for a tall thin triangle, a slight rotation
can affect its flatness dramatically.

%
%

\begin{lemma}
[Consequences of close base points.]
\label{lem-closebase}
Assume that $\objfun \in \calF$ and that 
$\Qscr$ is a compact set that does not contain the minimizer
of $\objfun$.  Let $\base_1$ and $\base_2$ denote two points in $\Qscr$,
and $\triang$ denote an RNM triangle contained in $\Qscr$.
For $i\in \{1,2\}$,
let $\wtrans_i$, $\htrans_i$, and $\flatness_i$
be the transformed width, height, and flatness of $\triang$
measured in the local coordinate frame
$\coordframe(\base_i)$ associated with $\base_i$,
and let $M_i$ be the matrix of Lemma~\ref{lem-affdef}
associated with $\coordframe(\base_i)$.
\begin{enumerate}[\upshape (i)]
\item\label{I:closebase matrix}
Given $\epsilon > 0$, there exists $\delta > 0$ 
(independent of $\base_1$ and $\base_2$) such that
if $\norm{\base_2 - \base_1} < \delta$, then
\[
    \norm{M_2\drop M\inv_1 - I} < \epsilon.
\]
\item\label{I:closebase area}
Given $\epsilon > 0$, there exists $\delta>0$ 
(independent of $\base_1$, $\base_2$, and $\triang$)
such that
if $\norm{\base_1 - \base_2} < \delta$, then
\begin{equation}
\label{eqn-arearel}
    (1-\epsilon)\Atrans_1 < \Atrans_2 <  (1+\epsilon)\Atrans_1.
\end{equation}
\item\label{I:closebase flatness}
Given $C,\epsilon > 0$, there is $\delta > 0$ 
(independent of $\base_1$, $\base_2$, and $\triang$)
such that if $\norm{\base_1 - \base_2} < \delta$
and $\wtrans_1 > C\htrans_1$, then
\begin{equation}
(1-\epsilon) \flatness_1 < \flatness_2 < (1+\epsilon)\flatness_1.
\label{eqn-flatrel}
\end{equation}
\end{enumerate}
\end{lemma}

\begin{proof}\hfill
\begin{enumerate}[\upshape (i)]
\item
We have
\[
    \norm{M_2\drop M\inv_1 - I} 
	= \norm{M_2 (M\inv_1 - M\inv_2)}
	\le
    \norm{M_2\drop}\;
    \norm{M\inv_1 - M\inv_2}.
\]
By Lemma~\ref{lem-affdef}\eqref{I:uniformly continuous},
the first factor $\norm{M_2\drop}$ is uniformly bounded,
and $M^{-1}$ is uniformly continuous as a function of $\base \in \Qscr$,
so the second factor $\norm{M\inv_1 - M\inv_2}$ 
can be made as small as desired
by requiring $\norm{\base_2-\base_1}$ to be small.
\item
Letting $\bptrans_2$ and $\bptrans_1$ denote the transformed
versions of a point $\bp$ in $\Qscr$ using $\coordframe(\base_1)$ and
$\coordframe(\base_2)$, we have
\begin{equation}
\label{eqn-bptransrel}
    \bptrans_2 = M_2\inv M_1\drop \bptrans_1 + M_2\inv(\bb_1 - \bb_2),
\end{equation}
so that $\bptrans_2$ and $\bptrans_1$ are
related by an affine transformation with matrix $M_2\inv M_1\drop$.
When an affine transformation with nonsingular matrix $B$
is applied to the vertices of a triangle, the area of the transformed
triangle is equal to the area of the original triangle
multiplied by $|\det(B)|$ \cite[page 144]{Klein}.
Applying this result to $\triang$ gives
\begin{equation}
\label{eqn-arearell}
    {\Atrans_2} = \Atrans_1\;|\det (M_2\inv M_1\drop)|.
\end{equation}
Since $|\det B|$ is a continuous function of $B$,
the result follows from \eqref{I:closebase matrix}.

\item
Because of~\eqref{I:closebase area},
it suffices to prove the analogous inequalities for width
instead of flatness.  Fixing 
two vertices of $\Delta$, we
let $\boldv_i$ denote the vector from one to the other
measured in $\coordframe(\base_i)$, and let $x(\boldv_i)$ denote
the corresponding $x$-component.
Then $|\boldv_i| \le \wtrans_1 + \htrans_1 = O(\wtrans_1)$,
since $\wtrans_1 > C \htrans_1$.
By \eqref{eqn-bptransrel}, $\boldv_2=M_2^{-1} M_1 \boldv_1$, 
so 
\[
	|x(\boldv_2)-x(\boldv_1)|
	\le |\boldv_2-\boldv_1| 
	= |(M_2^{-1} M_1 - I) \boldv_1| 
	= O(\norm{M_2^{-1} M_1-I} \cdot |\wtrans_1|).
\]
This bounds the change in $x$-component of each vector of the triangle
in passing from $\coordframe(\base_1)$ to $\coordframe(\base_2)$, and
it follows that
\[
	|\wtrans_2-\wtrans_1| = O(\norm{M_2^{-1} M_1-I} \cdot |\wtrans_1|).
\]
Finally, by~\eqref{I:closebase matrix}, $\norm{M_2^{-1} M_1-I}$
can be made arbitrarily small.\qedhere
\end{enumerate}
\end{proof}

%
%

\subsection{The contradiction hypothesis and the limiting level set}
\label{sec-levelset}

Our proof of Theorem~\ref{thm-rnmconv} is by contradiction.
Therefore we assume the following hypothesis 
for the rest of Section~\ref{sec-convproof}
and hope to obtain a contradiction.

%
\begin{hypothesis}\label{hyp1}
Assume that the RNM algorithm is applied to $\objfun \in \calF$
and a nondegenerate initial triangle,
and that it does not converge to the minimizer of $\objfun$.
\end{hypothesis}

We begin with a few easy consequences of Hypothesis~1.
Let $\triang_k$ be the RNM triangle at the start of the $k^{\textup{th}}$
iteration.
Let $\triangtrans_k$ be that triangle in the coordinate frame
determined by any one of its vertices, and
define its width $\wtrans_k$, height $\htrans_k$, and flatness $\Gamma_k$
as in Definition~\ref{def-flatness}.
%
\begin{lemma}
\label{L:consequences of Hypothesis 1}
Assume Hypothesis~1.
Then:
\begin{enumerate}[\upshape (a)]
\item\label{I:diameter tends to 0}
The diameter of $\triang_k$ tends to $0$.
\item\label{I:at least one limit point}
The RNM triangles have at least one limit point $\bplimit$.
\item\label{I:function values tend to limit}
The function values at the vertices of $\triang_k$ are
greater than or equal to $\objfun(\bplimit)$,
and they tend to $\objfun(\bplimit)$.
\item\label{I:action} 
If $\Qscr$ is a neighborhood of the level set of $\bplimit$,
then all the action of the algorithm is eventually inside $\Qscr$.
\item\label{I:eigenvalues of Hessian}
We may choose $\Qscr$ to be a compact neighborhood not containing
the minimizer of $\objfun$; 
then there is a positive lower bound on the smallest eigenvalue 
of the Hessian in $\Qscr$.
\item\label{I:transformed diameter}
The diameter of $\triangtrans_k$ tends to zero.
\item\label{I:transformed width and height}
We have $\wtrans_k \to 0$ and $\htrans_k \to 0$.
\end{enumerate}
\end{lemma}

\begin{proof}\hfill
\begin{enumerate}[\upshape (a)]
\item This follows from Lemma~\ref{lem-diamtozero}, even without Hypothesis~1.
\item Lemma~3.3 of \cite{LRWW98} states that the best,
next-worst, and worst function values in each
successive triangle cannot increase, and that at
least one of them must strictly decrease at each iteration.
Because level sets are bounded, compactness guarantees that there
is a limit point $\bplimit$.
\item This follows from the monotonic decrease in function values,
the shrinking of the diameter to zero, and the continuity of $\objfun$.
\item Since the level sets are compact, there is a compact neighborhood
$I$ of $f(\bplimit)$ such that $f^{-1}(I)$ is a compact set
contained in the interior of $\Qscr$.
By~\eqref{I:function values tend to limit}, the triangles are eventually
contained in $f^{-1}(I)$.
By~\eqref{I:diameter tends to 0},
eventually even the rejected points tested in each iteration
lie within $f^{-1}(I)$.
\item
The first statement follows since the minimizer is not on the 
level set of $\bplimit$.
The second statement follows from uniform continuity of the Hessian.
\item 
By Lemma~\ref{lem-affdef}\eqref{I:uniformly continuous},
the distortion of the triangles is uniformly bounded.
\item 
This follows from~\eqref{I:transformed diameter}.
\end{enumerate}
\end{proof}

{}For the rest of Section~\ref{sec-convproof}, 
we may assume that all our RNM triangles and test points 
lie in a compact set $\Qscr$ not containing the minimizer, as in 
Lemma~\ref{L:consequences of Hypothesis 1}\eqref{I:eigenvalues of Hessian}.
In particular, the implied bounds in Lemma~\ref{lem-affdef} are uniform.

%
%

\subsection{Flattening of the RNM triangles}
\label{sec-flattenout}

Under Hypothesis 1, we now show that
the transformed RNM triangles ``flatten out'' in the sense that the
height becomes arbitrarily small relative to the width.
The proof is again a proof by contradiction,
showing that, unless the
triangles flatten out, there must be a sequence of consecutive
reflections in which the value of $\objfun$ at the reflection point
is eventually less than $\objfun(\bplimit)$, contradicting 
Lemma~\ref{L:consequences of Hypothesis 1}\eqref{I:function values tend to limit}.

%
%

\begin{lemma}[Flattening of RNM triangles.]
\label{lem-rattozero}
Assume Hypothesis~1.
Then $\lim_{k \to \infty} \htrans_k/\wtrans_k = 0$.
\end{lemma}

\begin{proof}
Assume that the result of the lemma does not hold.
In other words, within the rest of this proof,
the following hypothesis is assumed:

%
%

\begin{hypothesis}\label{hyp2}
There exists $\flatconst > 0$
such that for arbitrarily large $k$ 
we have $\htrans_k/\wtrans_k > \flatconst$.
\end{hypothesis}
We may assume also that $\bplimit$ is a limit point of the
triangles $\triang_k$ for which $\htrans_k/\wtrans_k > \flatconst$.

Given $\epsilon>0$, we
define a {\sl downward-pointing sector\/} of points $(\xtrans,\ytrans)$ 
satisfying $\ytrans \le  \epsilon -\flatconst|\xtrans|/10$,
and a {\sl truncated sector} of points in the downward sector
that also satisfy $\ytrans \ge -\epsilon$:
see Figure~\ref{fig-sectors}.

%
%
\begin{figure}[htb]
\epsfxsize = 3in
\epsfysize = 1.5in
$$\epsfbox{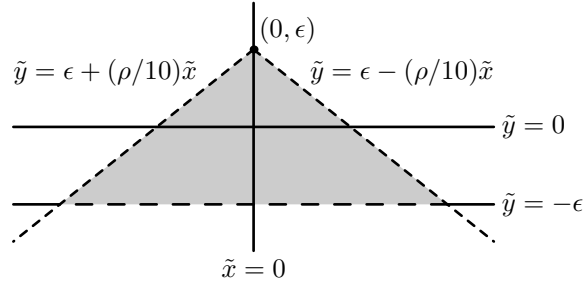}$$
\caption{{\small The downward-pointing sector lies between
the two finely dashed lines.  The truncated sector consists of
the shaded area, for $\flatconst = 8$ and $\epsilon = 0.5$.}}
\label{fig-sectors}
\end{figure}

We now show that there exists $\epsilon>0$ 
(depending on $\objfun$ and $\rho$)
such that, for any sufficiently advanced iteration $k_0$ 
for which $\htrans_{k_0}/\wtrans_{k_0} > \flatconst$,
\begin{enumerate}[\upshape (a)]
\item\label{I:initial triangle}
$\triangtrans_{k_0}$ is contained in the truncated sector.
\item\label{I:subsequent triangles}
If $\triangtrans$ 
is {\sl any} RNM triangle 
in the coordinates $(\xtrans,\ytrans)$ of $\coordframe_{k_0}$
such that
$\triangtrans$ is contained in the truncated sector
and has (transformed) width $\wtrans$ and height $\htrans$ 
satisfying $\htrans/\wtrans > \flatconst$,
then
\begin{enumerate}[\upshape (i)]
\item\label{I:it reflects} 
One RNM iteration reflects $\triangtrans$ 
to a new triangle $\triangtrans'$.
(And $\triangtrans'$ has the same width and height as $\triangtrans$, 
by Lemma~\ref{lem-refeffects}.)
\item\label{I:centroid drops} 
The $\ytrans$-coordinate of the centroid of $\triangtrans'$
is at least $88\htrans/300$ below that of $\triangtrans$.
\item\label{I:stays in sector}
$\triangtrans'$ is contained in the downward-pointing sector.
\item\label{I:function value}
If $\triangtrans'$ is not contained in the truncated sector,
then the function value at the new vertex is less than $\objfun(\bplimit)$.
\end{enumerate}
\end{enumerate}
Starting from~\eqref{I:initial triangle}, 
applying~\eqref{I:subsequent triangles} repeatedly
shows that the triangle in the $(\xtrans,\ytrans)$ coordinates
reflects downward until it exits the truncated sector through the bottom,
at which point the function value at the exiting vertex
is less than $\objfun(\bplimit)$,
which contradicts 
Lemma~\ref{L:consequences of Hypothesis 1}\eqref{I:function values tend to limit}.
Thus it remains to prove
\eqref{I:initial triangle} and~\eqref{I:subsequent triangles}.

\bigskip
\noindent
{\sl Proof of~\eqref{I:initial triangle}.\/}

By definition of $\coordframe_{k_0}$,
the point $(0,0)$ is a vertex of $\triangtrans_{k_0}$.
For any given $\epsilon>0$,
if $k_0$ is sufficiently large,
then 
Lemma~\ref{L:consequences of Hypothesis 1}\eqref{I:transformed diameter}
shows that the diameter of $\triangtrans_{k_0}$ is less than 
the distance from $(0,0)$ to the boundary of the truncated sector,
so $\triangtrans_{k_0}$ is entirely contained in the truncated sector.

\bigskip
\noindent
{\sl Proof of~\eqref{I:subsequent triangles}.\/}

Suppose that $\triangtrans$ is contained in the truncated sector
and satisfies $\htrans/\wtrans>\rho$.
Its vertices $\bptrans_i=(\xtrans_i,\ytrans_i)$ are the transforms 
of vertices $\bp_i$ of some $\triang$.
We will use the notation $\objfun_i=\objfun(\bp_i)$ for any subscript $i$,
and use similar abbreviations for other functions and coordinates.

We show first that 
the difference in $\objfun$ values at any two vertices $\bp_i$ and $\bp_j$ 
is within $3\htrans/100$ of the differences of their $\ytrans$-coordinates.
Using~\eqref{eqn-objtransform}, we find that
\begin{equation}
  \label{E:difference in f values}
    \objfun_i - \objfun_j  =  \ytrans_i - \ytrans_j
    + \half(\xtrans_i^2 - \xtrans_j^2)
    + \remain_i - \remain_j.
\end{equation}
The quantity $|\xtrans_i^2 - \xtrans_j^2|$ is bounded
by $2\wtrans|\xtrans_i| + \wtrans^2$.
If $\epsilon < \flatconst^2 /4000$,
then $|\xtrans| < \flatconst/200$ for any point in the truncated sector.
By Lemma~\eqref{L:consequences of Hypothesis 1}\eqref{I:transformed diameter},
if $k_0$ is large enough, then $\wtrans < \flatconst/100$.
It follows that
\begin{equation}
\label{eqn-xtrel}
\wtrans |\xtrans| < \frac{\wtrans\flatconst}{200} <
\frac{\htrans}{200}
\qquad\hbox{and}\qquad
\wtrans^2 < \frac{\flatconst\wtrans}{100} < \frac{\htrans}{100},
\end{equation}
so $|\xtrans_i^2 - \xtrans_j^2| < \htrans/50$.
On the other hand, $\remain_i - \remain_j$ is the line integral
of $(d\remain/d\xtrans,d\remain/d\ytrans)$
over a path of length at most $\wtrans+\htrans = O(\htrans)$.
Since $d\remain/d\xtrans$ and $d\remain/d\ytrans$ 
are $O(\max(|\xtrans|, |\ytrans|))$,
the derivatives can be made arbitrarily small on the truncated sector
by choosing $\epsilon$ small enough,
and we may assume that 
$|\remain_i - \remain_j| < \htrans/100$.
Now \eqref{E:difference in f values} yields
\begin{equation}
\objfun_i - \objfun_j = \ytrans_i - \ytrans_j + \zeta,
\quad\hbox{with}\quad
|\zeta| < \frac{3\htrans}{100}.
\label{eqn-funvalrel}
\end{equation}

\begin{enumerate}[(i)]
\item 
Let $\bp\subbest, \bp\subnext, \bp\subworst$ be the vertices of $\triang$ 
ordered so that 
$\objfun\subbest \le \objfun\subnext \le \objfun\subworst$.
Let $\bp\sublow, \bp\subymid, \bp\subhigh$ be the same vertices
ordered so that
$\ytrans\sublow \le \ytrans\subymid \le \ytrans\subhigh$.
Recall that the reflection point
$\bp\subref = \bp\subbest + \bp\subnext - \bp\subworst$
is accepted only if $\objfun\subref < \objfun\subnext$.
Equation~\eqref{eqn-funvalrel} implies that
$\ytrans\subbest$, $\ytrans\subnext$, $\ytrans\subworst$
are within $3\htrans/100$ of 
$\ytrans\sublow$, $\ytrans\subymid$, $\ytrans\subhigh$,
respectively.
Hence the difference 
\[
	\ytrans\subnext - \ytrans\subref 
	= \ytrans\subworst - \ytrans\subbest
\]
is within $6\htrans/100$ of $\ytrans\subhigh-\ytrans\sublow=\htrans$.
Applying \eqref{eqn-funvalrel} to the reflected triangle 
shows that $\objfun\subnext > \objfun\subref$, and
the reflection point is accepted.

\item
The reflection decreases 
the $\ytrans$ coordinate of the reflected vertex by
\[
	\ytrans\subworst-\ytrans\subref
	= 2\ytrans\subworst - \ytrans\subbest - \ytrans\subnext,
\]
which is within $4(3\htrans/100)$ of
\[
	2\ytrans\subhigh - \ytrans\sublow - \ytrans\subymid
	\ge \ytrans\subhigh - \ytrans\sublow = \htrans.
\]
Consequently, $\ytrans\subworst-\ytrans\subref \ge 88\htrans/100$,
and the centroid drops by at least $88\htrans/300$.

\item
Furthermore, 
$\xtrans\subref$ differs from $\xtrans\subworst$ by no more than
$2\wtrans$, i.e.,
$|\xtrans\subref| \le |\xtrans\subworst| + 2\wtrans$.
Since $\bp\subworst$ lies in the truncated sector
and $\flatconst\wtrans < \htrans$, it follows that
\begin{eqnarray*}
    \ytrans\subref + \frac{\flatconst}{10}|\xtrans\subref|
   &\; \le \;& \ytrans\subworst  - \frac{88\htrans}{100}
    + \frac{\flatconst}{10} (|\xtrans\subworst| + 2\wtrans) <
     \ytrans\subworst  - \frac{88\htrans}{100}
     + \frac{\flatconst}{10}|\xtrans\subworst|
     + \frac{2\htrans}{10} \\
    & \; < \; &  \ytrans\subworst +
\frac{\flatconst}{10} |\xtrans\subworst|
        \; < \; \epsilon.
\end{eqnarray*}
Thus, using the local coordinate frame $\coordframe_{k_0}$,
the reflection point $\bp\subref$
lies in the downward-pointing sector, and also lies in the
truncated sector as long as $\ytrans\subref \ge -\epsilon$.

\item
Let $\bb$ denote the base point of $\coordframe_{k_0}$, 
so $\trans\bb = (0,0)$.
For $\bptrans$ on the bottom edge of the truncated sector,
we have $\ytrans=-\epsilon$
and $\xtrans=O(\epsilon)$ as $\epsilon \to 0$ (similar triangles). 
Relation~\eqref{eqn-objtransform} then implies
\begin{equation}
\label{E:function along bottom edge}
	\objfun(\bp)  = \objfun(\bb) - \epsilon + O(\epsilon^2).
\end{equation}
Fixing $\epsilon$ to be small enough that $\objfun(\bp) - \objfun(\bb) < 0$ 
everywhere on the bottom edge,
we can also fix a neighborhood $U$ of the bottom edge
and a neighborhood $V$ of $\trans\bb = (0,0)$
such that $\objfun(\bp)<\objfun(\bb')$ holds 
whenever $\bptrans \in U$ and $\trans\bb'\in V$.

If $\triangtrans'$ is not in the truncated sector,
its new vertex $\bptrans\subref$ is within $\wtrans+\htrans$ 
of the bottom edge.
If $k_0$ is sufficiently large to make $\wtrans+\htrans$ small enough,
it follows that $\bptrans\subref \in U$.

By choice of $\bplimit$ (defined immediately following Hypothesis~2),
$k_0$ can be taken large enough that $\bplimit$ 
is arbitrarily close to $\bb$
in {\sl untransformed\/} coordinates.  
By Lemma~\ref{lem-affdef}\eqref{I:uniformly continuous},
the matrix defining the local coordinate transformation is bounded and
nonsingular.
Hence we can make $\bptranslimit$ arbitrarily close to $(0,0)$
in transformed coordinates, and in particular we can 
guarantee that $\bptranslimit$ lies in $V$.

Thus $\objfun(\bp\subref)<\objfun(\bplimit)$.\qedhere
\end{enumerate}
\end{proof}

%
%
\begin{remark}
\label{R:width greater than height}
An important consequence of Lemma~\ref{lem-rattozero}
is that $\wtrans>\htrans$ for $\triang_k$
measured in a coordinate frame associated to any one of its vertices,
so that Lemma~\ref{lem-closebase}\eqref{I:closebase flatness}
can be applied with $C=1$.
\end{remark}

%
%

\subsection{The distance travelled during a sequence of reflections}

We now show that a sequence of valid reflections, starting
from a sufficiently advanced iteration, does not move the triangle far.
This result limits the possible change in flatness caused by
moving the base point of the
local coordinate system from
the first to last triangle in the series of reflections.

%
%

\begin{lemma}
\label{lem-must-contract}
Assume Hypothesis~1.
Given $\movedist > 0$,
the following is true for any sufficiently large $k_0$ and any $k \ge k_0$:
if all steps taken by the RNM algorithm from $\triang_{k_0}$
to $\triang_k$ are reflections,
then the distance between the
transformed centroids of $\triang_{k_0}$ and $\triang_k$
is less than $\movedist$ (where we use a coordinate frame
whose base point is a vertex of $\triang_{k_0}$).
\end{lemma}

\begin{proof}
We work in the coordinates $(\xtrans,\ytrans)$ of $\coordframe_{k_0}$.
It suffices to show
that for sufficiently small positive $\epsilon < \kappa/2$,
if $k_0$ is sufficiently large and $\triangtrans$ is a later RNM triangle
with centroid in the box $\{|\xtrans| \le \epsilon, |\ytrans| \le \epsilon\}$,
then the next move does not reflect $\triangtrans$ 
so that its centroid exits the box.
More precisely, for suitable $\epsilon$ and $k_0$, the idea is to prove:
\begin{enumerate}[\upshape (a)]
\item\label{I:top of box} 
The centroid cannot escape out the top of the box (i.e., the $\ytrans$-coodinate
cannot increase beyond $\epsilon$) because the function value
of the reflection point would exceed the function values of $\triang_{k_0}$
(i.e., the function values near the center of the box).
\item\label{I:bottom of box} 
The centroid cannot escape out the bottom because the function value
there would be less than the limiting value $\objfun(\bplimit)$.
\item\label{I:sides of box} 
The centroid cannot escape out either side, 
because the triangle $\triangtrans$ will be flat enough
that the function values there are controlled mainly 
by the $\xtrans$-coordinates, which force the triangle to reflect inward
towards the line $\xtrans=0$.
\end{enumerate}
The conditions on $\epsilon$ and $k_0$ will be specified in
the course of the proof.

\bigskip
\noindent
{\sl Proof of~\eqref{I:top of box}.\/}

We copy the argument used in proving \eqref{I:subsequent triangles}(iv) 
of Lemma~\ref{lem-rattozero}.
Let $\bb$ be the base point used to define $\coordframe_{k_0}$.
For $\bp$ along the top edge of the box, by definition $\ytrans=\epsilon$.
Thus
the same argument that proved \eqref{E:function along bottom edge}
shows that
\[
	\objfun(\bp)  = \objfun(\bb) + \epsilon + O(\epsilon^2),
\]
and that if $\epsilon$ is sufficiently small, 
then there are neighborhoods $U$ of the top edge
and $V$ of $(0,0)$ such that $\objfun(\bp) > f(\bb')$ holds
whenever $\bptrans \in U$ and $\trans\bb'\in V$.
If $k_0$ is sufficiently large, and $\triangtrans$ is the later
triangle whose centroid is about to exit the box through the top,
then by
Lemma~\eqref{L:consequences of Hypothesis 1}\eqref{I:transformed diameter}, 
$\triangtrans_{k_0}$ and $\triangtrans$ are small enough
that $\triangtrans_{k_0} \subset V$ and $\triangtrans \subset U$,
so the function values at vertices of $\triang$
are greater than those for $\triang_{k_0}$,
which is impossible since function values at vertices of
successive RNM triangles
are non-increasing.

\bigskip
\noindent
{\sl Proof of~\eqref{I:bottom of box}.\/}

This case is even closer to the proof of \eqref{I:subsequent triangles}(iv) 
in Lemma~\ref{lem-rattozero}.
That argument shows that 
if $\epsilon$ is sufficiently small and $k_0$ is sufficiently large,
then the function values at the vertices of a triangle $\triang$
whose transformed centroid is about to exit through the bottom are strictly
less than the value $f(\bplimit)$ (which is made arbitrarily close
to $f(\bb)$ by taking $k_0$ large).
This contradicts Lemma~\ref{L:consequences of Hypothesis 1}\eqref{I:function values tend to limit}.

\bigskip
\noindent
{\sl Proof of~\eqref{I:sides of box}.\/}

By symmetry, suppose that $\triangtrans$ reflects 
so that its centroid exits the box through the {\sl right} side.
By Lemma~\ref{L:consequences of Hypothesis 1}\eqref{I:transformed width and height} and Lemma~\ref{lem-rattozero}, we may take $k_0$ large enough that
\begin{equation}
\wtrans_{k_0} < 0.01\epsilon
\qquad\hbox{and}\qquad
    \htrans_{k_0} < 0.01 \epsilon \wtrans_{k_0}.
\label{eqn-hwsmall}
\end{equation}
The width $\wtrans$ and height $\htrans$ of $\triangtrans$
are the same as that of $\triangtrans_{k_0}$.
So all vertices of $\triangtrans$ satisfy
$0.99\epsilon<\xtrans<1.01\epsilon$
and $-1.01 \epsilon < \ytrans < 1.01\epsilon$.
Let $\bvtrans = (\xtrans,\ytrans)$ and
$\bvtrans' = (\xtrans + \deltax, \ytrans + \deltay)$
be two such vertices.

We claim that if $\deltax>\wtrans/10$, then $\objfun(\bv')>\objfun(\bv)$.
By~\eqref{eqn-objtransform}, 
\begin{align*}
\label{eqn-fundifdel}
    \objfun(\bv') - \objfun(\bv) 
	&=
     \xtrans \delta_x + \half \delta_x^2 + \delta_y
     + (\remain(\xtrans + \delta x,\ytrans + \delta y) -
      \remain(\xtrans,\ytrans)) \\
	&\ge (0.99\epsilon)(\wtrans /10) + 0 - \htrans
	- (o(\epsilon) \wtrans + O(\epsilon) \htrans)
	\quad\textup{(by integrating Lemma~\ref{lem-affdef}(ii))} \\
	&\ge 0.099 \epsilon \wtrans + 0 - \htrans - 0.001 \epsilon \wtrans - \htrans  \quad\textup{(if $\epsilon$ is sufficiently small)} \\
	&= 0.098 \epsilon \wtrans - 2 \htrans \\
	&> 0 \qquad\textup{(by the second inequality in~\eqref{eqn-hwsmall}).} \\
\end{align*}

Now we can mimic part of the proof of~\eqref{I:subsequent triangles}
in Lemma~\ref{lem-rattozero},
but in the horizontal rather than the vertical direction.
Let $\xtrans_\subbest,\xtrans_\subnext,\xtrans_\subworst$
be the $\xtrans$-coordinates of the vertices ordered 
by increasing function value,
and let $\xtrans_\subleft,\xtrans_\subxmid,\xtrans_\subright$
be the same $\xtrans$-coordinates in increasing order.
The previous paragraph shows that 
$\xtrans_\subbest,\xtrans_\subnext,\xtrans_\subworst$
are within $\wtrans/10$ of $\xtrans_\subleft,\xtrans_\subxmid,\xtrans_\subright$,
respectively.
The reflection decreases the $\xtrans$ coordinate of the reflected vertex by
\[
	\xtrans\subworst-\xtrans\subref
	= 2\xtrans\subworst - \xtrans\subbest - \xtrans\subnext,
\]
which is within $4(\wtrans/10)$ of
\[
	2\xtrans\subright - \xtrans\subleft - \xtrans\subxmid
	\ge \xtrans\subright - \xtrans\subleft = \wtrans,
\]
so the $\xtrans$ coordinate of the centroid decreases
instead of increasing beyond $\epsilon$ as hypothesized.
\end{proof}

%
%

\subsection{Conditions at an advanced contraction}

Assuming Hypothesis~1, we next show that,
whenever a contraction step is taken at a sufficiently
advanced iteration $k$, we have $\htrans_k\drop = O(\wtrans_k^2)$.
We stress the assumption that the base of the local coordinate frame
at iteration $k$ lies inside $\triang_k$.

%
%

\begin{lemma}
\label{lem-smallheight}
Assume Hypothesis~1.
If $k$ is sufficiently large
and a contraction step is taken at iteration $k$ 
(meaning that the reflection point was not accepted),
then the transformed height $\htrans$ and width $\wtrans$ of
$\triang_k$ in a coordinate frame with base point inside
$\triang_k$ must satisfy $\htrans \le 10 \wtrans^2$.
\end{lemma}

\begin{proof}
Given a base point of the local coordinate frame in $\triang_k$,
Lemma~\ref{lem-affdef} shows that the difference in
values of $\objfun$ at any two points $\bp$ and $\boldv$ is
\begin{equation}
    \objfun(\bp) - \objfun(\boldv) =
     \ytrans_{\bp} - \ytrans_{\boldv}
   + \half(\xtrans_{\bp}^2 - \xtrans_{\boldv}^2)
   +  \remain(\xtrans{_\bp},\ytrans_{\bp})
   - \remain(\xtrans_{\boldv}, \ytrans_{\boldv}).
\label{eqn-funvertform}
\end{equation}

For $i \in \{1,2,3\}$,
let $\bp_i$ be the $i^{\textup{th}}$ vertex of $\triang_k$,
and let $\bptrans_i$ be its transform in the local coordinate frame.
We assume throughout the proof that $\bp_3$ is the worst vertex.
Let $\bp\subref := \bp_1 + \bp_2 - \bp_3$ be the reflect point,
and let $\bptrans\subref$ be its transform.

The origin of the coordinate frame is inside $\triang_k$, so
$|\xtrans_i| \le \wtrans$ for $i=1,2,3$.  
The RNM triangles are flattening out (Lemma~\ref{lem-rattozero}),
and the flatness does not change very much when
measured using the coordinate frame with a nearby base point
(Lemma~\ref{lem-closebase}\eqref{I:closebase flatness}).  
Hence, if $k$ is large enough,
$\htrans \le \wtrans$, so
$|\ytrans_i| \le \wtrans$ for $i=1,2,3$.
Since $\bp_3$ is the worst vertex, $\objfun(\bp_3) - \objfun(\bp_1) \ge 0$.
Substituting \eqref{eqn-funvertform} and rearranging yields
\begin{equation}
  \label{eqn-firstdirfirst}
\ytrans_3-\ytrans_1 \ge
     \half(\xtrans_1^2 - \xtrans_3^2) +
   \remain(\xtrans_1,\ytrans_1) - \remain(\xtrans_3,\ytrans_3).
\end{equation}
Because $|\xtrans_i| \le \wtrans$ and $|\xtrans_j| \le \wtrans$,
we obtain $|\xtrans_i^2 - \xtrans_j^2| \le \wtrans^2$,
so the inequality~\eqref{eqn-firstdirfirst} implies
\begin{equation}
\label{eqn-firstdirsecond}
\ytrans_3-\ytrans_1 \ge
     -\half \wtrans^2  +
   \remain(\xtrans_1,\ytrans_1) - \remain(\xtrans_3,\ytrans_3).
\end{equation}

Next we use the definition of the
reflection point to obtain bounds in the other direction.
A contraction occurs only when the reflection point
is not accepted (see Step~3 of Algorithm RNM
in Section~\ref{sec-nmalgs}), which implies that
$\objfun(\bp\subref) -\objfun(\bp_2)  \ge 0$.
Substituting \eqref{eqn-funvertform} and rearranging yields
\begin{equation}
\label{eqn-seconddirfirst}
\ytrans\subref-\ytrans_2 \ge
     \half(\xtrans_2^2 - \xtrsubrsq) +
     \remain(\xtrans_2,\ytrans_2) - \remain(\xtrans\subref,\ytrans\subref).
\end{equation}
By definition of $\bp\subref$, we have
$\ytrans\subref - \ytrans_2 = \ytrans_1 - \ytrans_3$.
Substituting into the left-hand side of
\eqref{eqn-seconddirfirst} yields
\begin{equation}
  \label{E:another one}
\ytrans_1-\ytrans_3 \ge
     \half(\xtrans_2^2 - \xtrsubrsq) +
     \remain(\xtrans_2,\ytrans_2) - \remain(\xtrans\subref,\ytrans\subref).
\end{equation}
We have $|\xtrans_2| \le \wtrans$
and $|\xtrans\subref-\xtrans_2| = |\xtrans_1-\xtrans_3| \le \wtrans$,
so 
\[
	|\xtrans_2^2 - \xtrsubrsq| = |\xtrans_2+\xtrans\subref| \cdot
|\xtrans_2-\xtrans\subref| \le 3\wtrans^2,
\]
and substituting into~\eqref{E:another one} yields
\begin{equation}
\label{eqn-seconddirsecond}
\ytrans_1-\ytrans_3  \ge
     -\threehalves \wtrans^2 +
   \remain(\xtrans_2,\ytrans_2) - \remain(\xtrans\subref,\ytrans\subref).
\end{equation}
 
If $k$ is sufficiently large, we know from
Lemmas~\ref{lem-diamtozero} and \ref{lem-affdef}
that, in the smallest box containing a transformed advanced
RNM triangle and its reflection point, 
$|d\remain/d\xtrans| \le \wtrans$ and $|d\remain/d\ytrans| \le \half$.
Consequently,
\begin{align}
\label{eqn-remainrels}
|\remain(\xtrans_1,\ytrans_1) - \remain(\xtrans_3,\ytrans_3)| & \le 
\wtrans|\xtrans_1 - \xtrans_3| + \half|\ytrans_1 - \ytrans_3| 
\le \wtrans^2 + \half|\ytrans_1 - \ytrans_3| \\
|\remain(\xtrans_2,\ytrans_2) - \remain(\xtrans\subref,\ytrans\subref)|
& \le 
\wtrans|\xtrans_1 - \xtrans_3| + \half|\ytrans_1 - \ytrans_3|
\le \wtrans^2 + \half|\ytrans_1 - \ytrans_3|.
\nonumber
\end{align}
Substituting the equations~\eqref{eqn-remainrels} into
\eqref{eqn-firstdirsecond} and~\eqref{eqn-seconddirsecond},
respectively, we obtain
\begin{equation}
\label{eqn-togetherfirst}
\ytrans_3-\ytrans_1 \ge
     -\threehalves \wtrans^2 - \half|\ytrans_1 - \ytrans_3|
\quad\hbox{and}\quad
\ytrans_1-\ytrans_3  \ge
     -\fivehalves\wtrans^2 - \half|\ytrans_1 - \ytrans_3|.
\end{equation}
These imply $\ytrans_3 - \ytrans_1 \ge  -3\wtrans^2$
and $\ytrans_1 - \ytrans_3 \ge -5\wtrans^2$,
so $|\ytrans_1-\ytrans_3| \le 5 \wtrans^2$.
Our numbering of $\bp_1$ and $\bp_2$ was arbitrary,
so $|\ytrans_2-\ytrans_3| \le 5 \wtrans^2$ too.
These two inequalities imply $\htrans \le 10 \wtrans^2$.
\end{proof}

%
%
\begin{remark}
The lemma just proved applies
to an RNM triangle not at an arbitrary iteration, but only
at a sufficiently advanced iteration $k$.  
Even for large $k$,
the condition $\htrans \le 10\wtrans^2$ is necessary
but not sufficient to characterize an RNM triangle for which
a contraction occurs.
\end{remark}

Figures~\ref{fig-refheight} and \ref{fig-conheight} illustrate two
cases for the function $\half \xtrans^2 + \ytrans + \half \ytrans^2$.
The worst vertex is at the origin in each figure.
In Figure~\ref{fig-refheight}, we have $\htrans = 1.2 \times 10^{-6}$ and
$\wtrans = 2 \times 10^{-4}$, so $\htrans/\wtrans^2 = 30$;
as Lemma~\ref{lem-smallheight} would predict at an advanced iteration,
the triangle reflects instead of contracting.
In Figure~\ref{fig-conheight}, by contrast, 
$\htrans = 3 \times 10^{-8}$ and
$\wtrans = 2 \times 10^{-4}$, so $\htrans/\wtrans^2 = \frac{3}{4}$ and
an outside contraction is taken.
The vertical scale in each figure is greatly compressed
compared to the horizontal, and
the vertical scale in
Figure~\ref{fig-refheight} differs from that in
Figure~\ref{fig-conheight} by two orders of magnitude.

%
%
\begin{figure}[htb]
\begin{center}
{\includegraphics[width=5in,height=2in]{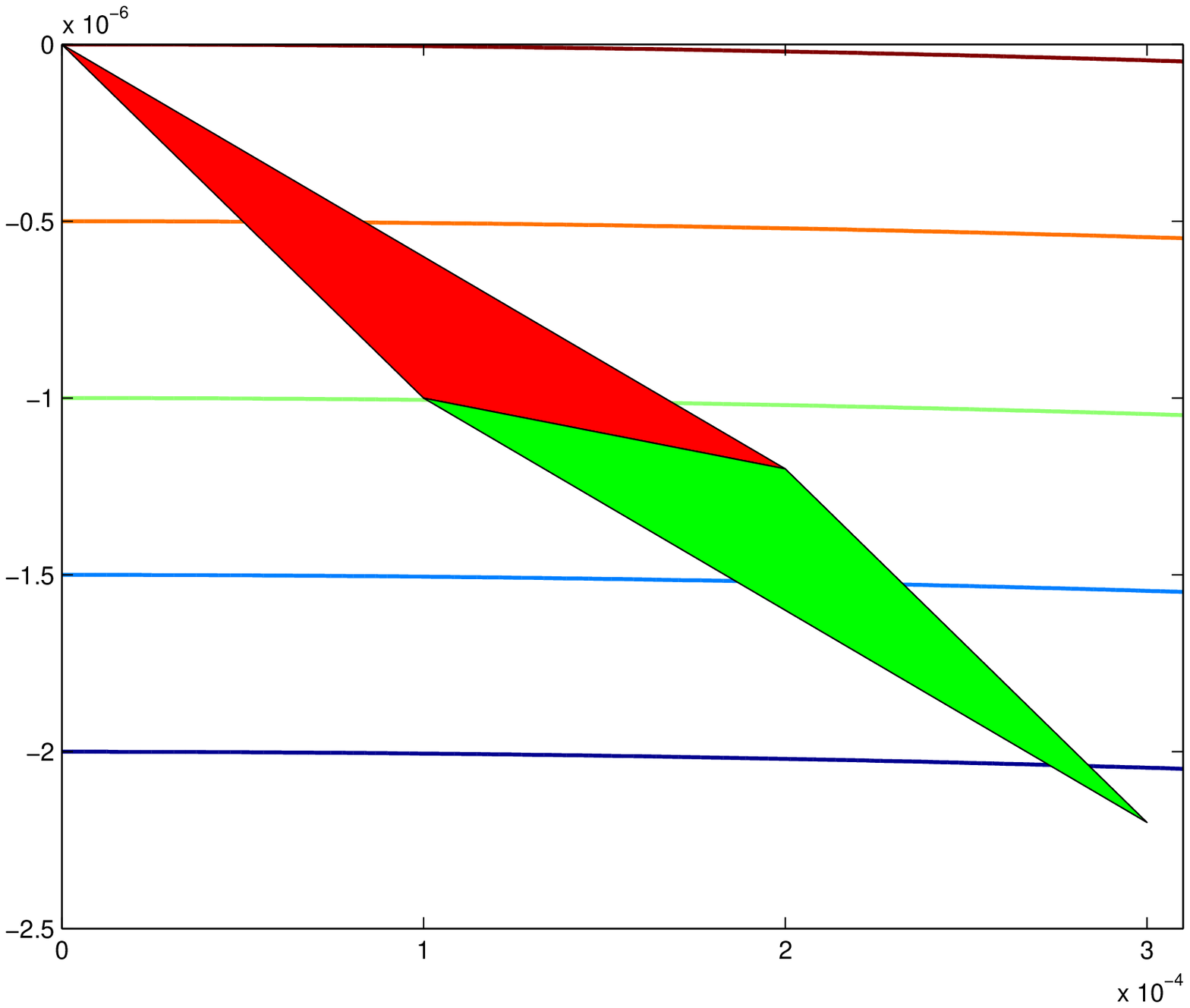}}
\caption{{\small The contours of
$\ytrans + \half \xtrans^2 + \half \ytrans^2$ are shown
along with an RNM triangle with $\htrans/{\wtrans}^2 = 30$.
The reflection is accepted.}}
\label{fig-refheight}
\end{center}
\end{figure}

%
%
\begin{figure}[htb]
\begin{center}
{\includegraphics[width=5in,height=2in]{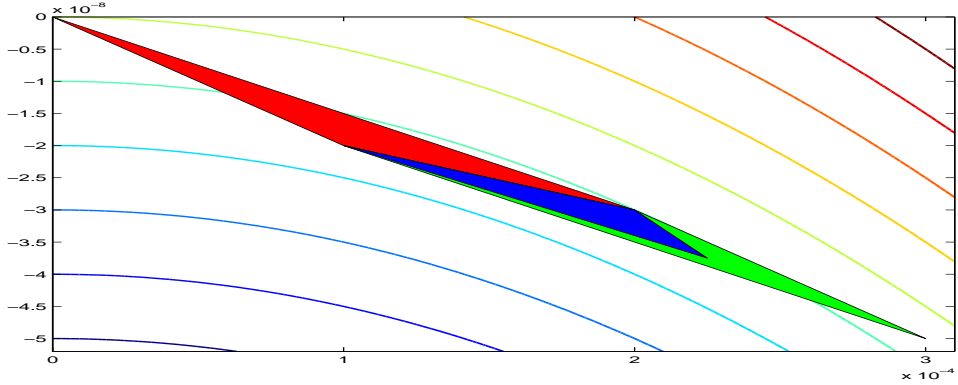}}
\caption{{\small The contours of
$\ytrans + \half \xtrans^2 + \half \ytrans^2$,
are shown along with an RNM triangle
with $\htrans/{\wtrans}^2 = \frac{3}{4}$.
The reflection step is not accepted, and an outside contraction is
performed.  Note the difference, by four
orders of magnitude, between the horizontal and vertical scales.}}
\label{fig-conheight}
\end{center}
\end{figure}

%
%

\begin{lemma}
\label{lem-smalldelta}
Under the assumptions of Lemma~\ref{lem-smallheight},
if $k$ is sufficiently large
and a contraction step is taken at iteration $k$, 
then $\flatness_k \le 10$, where $\flatness_k$ is
the flatness of $\triangtrans_k$ 
as in Definition~\ref{def-flatness}.
\end{lemma}

\begin{proof}
Let $\wtrans$, $\htrans$, $\Atrans$ be the width,
height, and area of $\triang_k$
with respect to the coordinate frame
associated by Lemma~\ref{lem-affdef} to a vertex of $\triang_k$.
If $k$ is sufficiently large, then
Lemma~\ref{lem-smallheight} implies $\htrans \le 10\wtrans^2$.  
Hence
\[
    \flatness_k = \frac{\Atrans}{\wtrans^3} \; \le \;
      \frac{\htrans \wtrans}{\wtrans^3 }
   \; \le \; \frac{(10\wtrans^2)\wtrans}{\wtrans^3} \; =\;  10.\qedhere
\]
\end{proof}

%
%
\subsection{Eliminating the impossible: increasing flatness is
unavoidable}\label{sec-elimimposs}

The final piece of the proof of Theorem~\ref{thm-rnmconv}
will show that, for sufficiently
advanced iterations, the flatness of the
RNM triangles must increase by a factor of at least $1.001$
{\sl within a specified number of
iterations following a contraction}.
To obtain this result, we begin by characterizing the
structure of RNM vertices at sufficiently advanced iterations
following a contraction, and then defining a related but
simpler triangle.

%
%
\subsubsection{A simpler triangle.}
\label{sec-simpletri}
Assume that (i) there
is a limit point $\bplimit$ of the RNM triangles that is not the
minimizer of $\objfun$, (ii) $k_0$ is sufficiently large, and
(iii) iteration $k_0$ is a contraction.
For the RNM triangle $\triang_{k_0}$, let $\coordframe_1$
denote the coordinate frame whose base point is
the vertex of $\triang_{k_0}$ with the worst value of $\objfun$:
\begin{equation}
\label{eqn-coordonedef}
\hbox{base}(\coordframe_1) = (\bp\subworst)_{k_0}.
\end{equation}
This first coordinate frame is used to identify
$\bptrans\subleft$ and $\bptrans\subright$, the
transformed vertices of $\triang_{k_0}$
with leftmost and rightmost $\xtrans$ coordinates.

A second coordinate frame, $\coordframe_2$,
is defined next whose base point (measured in frame $\coordframe_1$)
is the midpoint of $[\bptrans\subleft, \bptrans\subright]$:
\begin{equation}
\label{eqn-coordtwodef}
\hbox{base}(\coordframe_2) = \half(\bptrans\subleft +
\bptrans\subright).
\end{equation}
Unless otherwise specified, the coordinate frame
$\coordframe_2$ is used throughout the
remainder of this proof.
The base points of $\coordframe_1$ and $\coordframe_2$
will be  arbitrarily close if $k_0$ is sufficiently large.

We assume that
$k_0$ is sufficiently large so that the RNM triangles
have become tiny in diameter and
flattened out (Lemma~\ref{lem-rattozero}).
The reason for defining $\coordframe_2$ is that we can
choose a small $\eta > 0$ such that
the transformed three vertices of $\triang_{k_0}$,
measured in coordinate frame $\coordframe_2$, may be
expressed as
\begin{equation}
    {\ba}_0 = \mtx{c}{-\eta\\
                     -u\eta^2}, \quad
    {\bb}_0 = \mtx{c}{s\eta\\
                     t\eta^2}, \quad\hbox{and}\quad
    {\bc}_0 = \mtx{c}{\eta\\
                    u\eta^2},
\label{eqn-trik0verts}
\end{equation}
where vertex ${\ba}_0$ corresponds to $\bp\subleft$
and vertex ${\bc}_0$ to $\bp\subright$.

Without loss of generality the
value of $s$ in \eqref{eqn-trik0verts} can be taken as nonnegative.
The vertices ${\ba}_0$ and ${\bc}_0$ were leftmost and
rightmost when measured in $\coordframe_1$; by
Lemma~\ref{lem-closebase}\eqref{I:closebase matrix}, 
the $s$ in \eqref{eqn-trik0verts} cannot be too much larger than $1$.  
We assume that $k_0$ is large enough so that $0 \le s \le 1.00001$.

Because of the form of the vertices in \eqref{eqn-trik0verts}
and the bounds on $s$,
the transformed width $\wtrans$ of $\triang_{k_0}$ (measured using
coordinate frame $\coordframe_2$) can be no
larger than $2.00001\eta$.
Iteration $k_0$ is, by assumption, a contraction, so
it follows from Lemma~\ref{lem-smallheight} that
the transformed height of $\triang_{k_0}$ satisfies
$\htrans \le 10 \wtrans^2$, and hence
$\htrans \le 40.0005\eta^2$.  
Since $\htrans$ is equal to
the larger of $2|u|\eta^2$ or $(|u| + |t|)\eta^2$, it follows that
$|u| \le 40.0005$ and $|t| \le 40.0005$ in~\eqref{eqn-trik0verts}.

If $\triang$ and $\triang'$ are any two consecutive RNM
triangles in which the same coordinate frame is used,
the new vertex of $\triang'$ is a linear combination
of the vertices of $\triang$,
with rational coefficients defined by the choice of worst vertex
and the nature of the move.
(See \eqref{eqn-twodref}--\eqref{eqn-twodinc}.)
Furthermore, the values of $\wtrans$ and $\htrans$ in $\triang$
and $\triang'$ remain the same or decrease, and, if $\bv$ is
any vertex of $\triang$ and $\bv'$ is any vertex of $\triang'$,
then $|\xtrans_{\bv'} - \xtrans_{\bv}|\le 2\wtrans$
and $|\ytrans_{\bv'} - \ytrans_{\bv}|\le 2\htrans$.
Thus, after $\ell \ge 0$ moves, we reach a triangle $\triang_{k_0+\ell}$
for which each transformed vertex $\bvtrans$ has the form
\begin{equation}
\label{eqn-lammurels}
\bvtrans = \mtx{c}{\lambda\eta\\
                     \mu\eta^2},
\quad\hbox{where}\quad
|\lambda| \le 1.00001 + 4.00002\ell \;\;\hbox{and}\;\;
|\mu| \le 40.0005(1+2\ell).
\end{equation}

%
%
\subsubsection{Rescaled inequalities associated with RNM moves.}
\label{sec-rnminequalities}
The next step is 
to make a rescaling of coordinates to define a triangle ${\btriangell}$
that is related to $\triangtrans_{k_0+\ell}$
by the diagonal affine transformation $\diag(\eta,\eta^2)$.
Let $\bptrans = (\lambda\eta, \mu\eta^2)$ be a point
in $\triangtrans_{k_0 + \ell}$ measured in $\coordframe_2$.
Then
\begin{equation}
\label{eqn-triangrels}
\bptrans = \mtx{c}{\lambda\eta\\
                     \mu\eta^2} \quad
     \;\;\hbox{corresponds to}\quad
\bP = \mtx{c}{\lambda\\
               \mu}
\;\;\hbox{(a point in $\btriangell$)},
\end{equation}
where $\lambda$ and $\mu$ satisfy the bounds~\eqref{eqn-lammurels}.
The flatness of $\btriangell$,
defined as $\hbox{area}(\btriangell)/(\hbox{width}(\btriangell))^3$,
is equal to the flatness of $\triang_{k_0+\ell}$ measured in
coordinate frame $\coordframe_2$.

Assume now that $\ell\le 20$; the reason for this
limit on $\ell$ will emerge later in Proposition~\ref{prop-14steps}.
For vertex $i$ of $\triang_{k_0 + \ell}$, equation~\eqref{eqn-lammurels} 
shows that
the coefficients in its transformed coordinates satisfy
$|\lambda_i| < 82$ and $|\mu_i| < 3000$.
By \eqref{eqn-objtransform},
\eqref{eqn-remaindef}, and \eqref{eqn-lammurels}, once
$k_0$ is large enough to make $o(\eta^2)$ sufficiently small,
the difference in $\objfun$ values between vertices $i$ and $j$ is
\begin{eqnarray}
\nonumber
   \objfun(\bv_i) - \objfun(\bv_j) & = &
   \eta^2 [(\half\lambda_i^2 + \mu_i) - (\half\lambda_j^2 + \mu_j)]
   + \remain(\eta\mu_i,\eta^2\lambda_i) - \remain(\eta\mu_j,\eta^2\lambda_j^2)\\
\label{eqn-funsimplediff}
   & = & \eta^2 [ (\half\lambda_i^2 + \mu_i) - (\half\lambda_j^2 + \mu_j)]
   + o(\eta^2).
\end{eqnarray}

Let $\psi$ denote the simple quadratic function
\begin{equation}
    \psi(\lambda,\mu) \definedas \half \lambda^2 + \mu.
\label{eqn-psidef}
\end{equation}
Then \eqref{eqn-funsimplediff} shows that, if $k_0$ is large
enough, the following
relationships hold between $\objfun$ 
at vertices of $\triang_{k_0 + \ell}$ and
$\psi$ at vertices of ${\btriangell}$:
\begin{equation}
\label{eqn-fimplicationone}
   \objfun(\bv_i) \ge \objfun(\bv_j)  \quad\hbox{implies}\quad 
   \psi(\lambda_i,\mu_i) > \psi(\lambda_j,\mu_j) - 10^{-6},
\end{equation}
where $10^{-6}$ is not magical, but simply a number small enough
so our subsequent results follow.

%
%
\begin{example}
For illustration, let $\ell = 0$.  Based on~\eqref{eqn-trik0verts},
the vertices of ${\btriangzer}$ are given by
\begin{equation}
\label{eqn-btriangzerverts}
    {\bA}_0 = \mtx{c}{-1\\
                     -u}, \quad
    {\bB}_0 = \mtx{c}{s\\
                     t}, \quad\hbox{and}\quad
    {\bC}_0 = \mtx{c}{1\\
                    u},
\end{equation}
and suppose that ${\ba}_0$ is the worst transformed vertex
of $\triang_{k_0}$, i.e.\ that
$$
    \objfun({\ba}_0) \ge \objfun({\bb}_0) \quad\hbox{and}\quad
    \objfun({\ba}_0) \ge \objfun({\bc}_0).
$$
Application of~\eqref{eqn-fimplicationone} gives
$\psi(-1, -u) > \psi(s,t)  -10^{-6}$
and $\psi(-1, -u) > \psi(1,u)  -10^{-6}$, i.e.
$$
    \half -u  > \half s^2 + t - 10^{-6}\quad\hbox{and}\quad
    10^{-6} > 2u\;\;
\hbox{(a simplification of $\half -u  > \half + u - 10^{-6}$)}.
$$
\end{example}

In this way, inequalities characterizing the
transformed vertices \eqref{eqn-lammurels} of $\triang_{k_0+\ell}$ when
applying the RNM algorithm with function $\objfun$ can be derived in
terms of vertices of the simpler
triangle ${\btriangell}$ when applying the RNM algorithm to the
function $\psi(\lambda,\mu)$, except that
{\sl both possible outcomes of a comparison
must be allowed\/} if the two values of $\psi$ are within $10^{-6}$.
The importance of~\eqref{eqn-fimplicationone} is that, for
$\ell\le 20$, a possible sequence of RNM moves specifying the
move type and worst vertex leads to a set of algebraic
inequalities in $s$, $t$, and $u$.

%
%
\subsection{Flatness must increase after no more than 14 steps}
\label{sec-fourteensteps}

In the remainder of this section, we consider the transformed
width, area, and flatness
of a sequence of RNM triangles, $\triang_{k_0}$, \dots,
$\triang_{k_0 + \ell}$, defined using a coordinate
frame whose base point is in $\triang_{k_0}$.  Accordingly,
notation is needed that separately identifies
the RNM triangle being measured
and the relevant coordinate frame.
The value $\flatness_{k}^{(1)}$ will denote the flatness of RNM
triangle $\triang_{k}$ measured in $\coordframe_1$ of \eqref{eqn-coordonedef},
and $\flatness_{k}^{(2)}$ will denote the flatness of $\triang_k$
measured in $\coordframe_2$ \eqref{eqn-coordtwodef}, with similar notation for
$\wtrans$ and $\Atrans$.
Since the base points of coordinate frames
$\coordframe_1$ and $\coordframe_2$ 
are in $\triang_{k_0}$, an essential point is that, when $k> k_0$,
the triangle containing the base point of the coordinate frame
is different from the triangle being measured.

The result in the following proposition
was found using symbolic computation software.

%
%
\begin{proposition}
\label{prop-14steps}
Assume Hypothesis~1.
If $k_0$ is sufficiently large and a contraction step is taken
at iteration $k_0$, then there exists $\ell$ with $1 \le \ell \le 14$
such that $\flatness_{k_0+\ell}^{(2)} > 1.01\; \flatness_{k_0}^{(2)}$.
\end{proposition}

Before giving the proof, we sketch the basic idea. 
As just described in Section~\ref{sec-rnminequalities},
we are in a situation where two properties
apply: (1) the transformed objective function at the
scaled point $(\lambda,\mu)\T$ can be
very well approximated by the quadratic function 
$\psi(\lambda,\mu) \definedas \half \lambda^2 + \mu$ in~\eqref{eqn-psidef},
and (2) the RNM move sequences of interest can be analyzed by
beginning with an initial simplified (scaled) triangle whose
vertices (see~\eqref{eqn-btriangzerverts}) involve bounded
scalars $(s,t,u)$ that lie in a compact set.  Under these
conditions, the proof explains how algebraic constraints can be derived
that characterize geometrically valid sequences of RNM moves.
Further algebraic constraints involving $s$ can also be defined
that must be satisfied when the 
flatness increases by a factor of no more than $1.01$.

In principle, one could establish the result of the proposition
by numerically checking flatness for all geometrically
valid RNM move sequences beginning with the simplified triangle,
but this approach is complicated, structureless,
and too time-consuming for numerical
calculation.  Instead, we used Mathematica\texttrademark~7.0 to 
construct symbolic inequalities representing RNM move sequences
such that
\begin{itemize}
\item $s$, $t$, and $u$ are suitably bounded, 
\item the geometric condition~\eqref{eqn-fimplicationone} for a valid RNM move applies, and
\item the flatness increases by a factor of {\sl less than or equal to\/} $1.01$.
\end{itemize}

\begin{proof}[Proof of Proposition~\ref{prop-14steps}]
The flatness is not changed by a reflection step as long
as the same coordinate frame is retained.
Assuming that $k_0$ is sufficiently large and that the move taken
during iteration $k_0$ is a contraction, we wish to show that there is
an index $\ell$ satisfying
$1 < \ell\le 14$ such that the flatness $\flatness$ of
the RNM triangle $\triang_{k_0 + \ell}$, {\sl measured in
coordinate frame} $\coordframe_2$, must be a factor
of at least $1.01$ larger than the flatness of $\triang_{k_0}$,
i.e., that
\begin{equation}
\label{eqn-largeflatrat}
    \frac{\flatness_{k_0 +\ell}^{(2)}}{\flatness_{k_0}^{(2)}} =
\frac{\Atrans_{k_0 +\ell}^{(2)}}{\Atrans_{k_0}^{(2)}}
    \left(\frac{\wtrans_{k_0}^{(2)}}
    {\wtrans_{k_0 + \ell}^{(2)}}\right)^{\!\! 3}
    > 1.01.
\end{equation}

Let us prove~\eqref{eqn-largeflatrat} directly
for $\ell = 1$ when ${\bA}_0$ of~\eqref{eqn-btriangzerverts}
is the worst vertex of ${\btriangzer}$ and an inside contraction occurs.
In this case, the next triangle ${\btriangone}$ has vertices
\begin{equation} \label{eqn-vertices}
    {\bA}_1 = \mtx{c}{\fourth s - \fourth\\ \fourth t - \fourth u}, \quad
    {\bB}_1 = \mtx{c}{s\\
                     t}, \quad\hbox{and}\quad
    {\bC}_1 = \mtx{c}{1\\
                    u},
\end{equation}
where the first vertex $\bA_0$ has been
replaced.
We have two cases:
\begin{itemize}
\item
If $0 \le s \le 1$, then $\wtrans({\btriangzer})= 2$ 
and $\wtrans({\btriangone}) = \fivefourths - \fourth s \le \fivefourths$,
which implies
that $\wtrans({\btriangzer})/\wtrans({\btriangone}) \ge \frac{8}{5}$.
\item If  $1< s \le 1.00001$, then
$\wtrans({\btriangzer})\ge  2$ and 
$\wtrans({\btriangone}) = \threefourths s + \fourth$, so that
$\wtrans({\btriangone}) \le 1.0000075$ and 
$\wtrans({\btriangzer})/\wtrans({\btriangone}) \ge 1.9999$.
\end{itemize}
For all $s$ satisfying $0 \le s \le 1.00001$, it follows that
$\wtrans({\btriangzer})/\wtrans({\btriangone}) \ge \frac{8}{5}$,
and hence that
\[
  \left(\frac{\wtrans({\btriangzer})}
   {\wtrans({\btriangone})}\right)^{\!\! 3} 
   \ge \left({\frac{8}{5}}\right)^{\!\! 3} = 4.096.
\]
The area of ${\btriangone}$ is half the area of
${\btriangzer}$.  Hence the ratio of the flatnesses of ${\btriangone}$
and ${\btriangzer}$ satisfies
\[
  \frac{\flatness({\btriangone})}{\flatness({\btriangzer})} =
\frac{\Atrans({\btriangone})}{\Atrans({\btriangzer})}
  \left(\frac{\wtrans({\btriangzer})}
  {\wtrans({\btriangone})}\right)^{\!\! 3} 
 \ge \half(4.096) > 1.01.
\]
The same argument applies when
${\btriangone}$ is the result of an {\sl outside}
contraction in which vertex ${\bA}_0$ is the worst.

But when the sequence of moves begins with
a contraction in which vertex ${\bB}_0$ or ${\bC}_0$ is worst,
we must break into further cases,
and the analysis becomes too complicated to do by hand.
To examine such sequences of RNM moves, we use a
Mathematica program that generates
inequalities involving vertices of
$\btriangell$ and the function $\psi$ of~\eqref{eqn-psidef},
as described in
Section~\ref{sec-rnminequalities}.

Any sequence of RNM moves (where a move is specified by
the worst vertex and the type of move) starting with triangle
$\triang_{k_0}$ gives rise to a set of algebraic inequalities in
$s$, $t$, and $u$.
The $i^{\textup{th}}$ of these latter inequalities has one of the forms
$\phi_i(s) + \nu_i t + \omega_i u > \theta_i$
or $\phi_i(s) + \nu_i t + \omega_i u \ge \theta_i$,
where $\phi_i(s)$ is a quadratic polynomial in $s$ with rational
coefficients, and
$\nu_i$, $\omega_i$, and $\theta_i$ are rational constants.

The next step is to determine whether there are
acceptable values of $s$, $t$, and $u$ for which these inequalities
are satisfied.
To do so, we begin by treating $s$ as constant (temporarily)
and considering the feasibility of a system of {\sl linear\/}
inequalities in $t$ and $u$,
namely the system $Nz \ge d$, where $z = (t\;\; u)^T$,
the $i^{\textup{th}}$ row of $N$ is
$(\nu_i \;\; \omega_i)$, and $d_i = \theta_i - \phi_i(s)$.
A variant of Farkas' lemma \cite[page 89]{Schrijver}
states that the system of linear inequalities
$Nz \ge d$ is feasible if and only if $\gamma^T d \le 0$
for every vector
$\gamma$ satisfying $\gamma\ge 0$ and  $N^T \gamma = 0$.
If the only nonnegative vector $\gamma$ satisfying
$N^T \gamma = 0$ is $\gamma = 0$, then $Nz \ge d$ is feasible
for any $d$.

The existence (or not) of a nonnegative nonzero $\gamma$
in the null space of $N^T$ can be determined symbolically by noting
that the system $Nz \ge d$ is feasible if and only if
it is solvable for every subset of three rows of $N$.
Let $\widehat N$ denote the $3\times 2$ matrix consisting of
three specified rows
of $N$, with a similar meaning for $\widehat d$.
To determine the feasibility of $\widehat Nz \ge \widehat d$,
we first find a vector $\widehat\gamma$
such that ${\widehat N}^T \widehat\gamma = 0$.

If $\widehat N$ has rank $2$, then $\widehat\gamma$
is unique (up to a scale factor) and we can write
${\widehat N}^T$ (or
a column permutation) so that the leftmost $2\times 2$ submatrix
$B$ is nonsingular.  Then, with
$$
    {\widehat N}^T = \mtx{ccc}{\nu_1 & \nu_2 & \nu_3\\
                       \omega_1 & \omega_2 & \omega_3} =
    \mtx{cc}{B & h},
\quad\hbox{$\widehat\gamma$ is a multiple of}\;\;
    \mtx{c}{-B\inv h\\
              1},
$$
where the components of $B\inv$ and $h$ are rational numbers.
If (with appropriate scaling) $\widehat\gamma\ge 0$ with at least
one positive component,
then ${\widehat N}^T z \ge \widehat d$ is solvable
if and only if $\widehat\gamma^T {\widehat d} \le 0$.  If the
components of $\widehat\gamma$ do not have the same sign,
${\widehat N}^T z \ge \widehat d$ is solvable for
any $\widehat d$.

If ${\widehat N}$ has rank one, its three rows must be
scalar multiples of the same vector, i.e., the $i^{\textup{th}}$ row
is $(\beta_i \nu_1 \;\; \beta_i \omega_1)$,
and the null vectors of ${\widehat N}^T$ are linear
combinations of $(\beta_2, -\beta_1, 0)^T$,
$(0, \beta_3, -\beta_2)^T$, and
$(\beta_3,0,-\beta_1)^T$.

Since the components of $d$ are quadratic
polynomials in $s$ and the components of each $\widehat\gamma$
are rational numbers, the
conditions for feasibility of $Nz \ge d$ (e.g., the
conjunction of conditions that
${\widehat\gamma}^T \widehat d \le 0$ for each set of three
rows of $N$) can be expressed as a Boolean combination of
quadratic inequalities in $s$ with rational coefficients
that, for a given value of $s$, evaluates to ``True'' if and
only if there exist $t$ and $u$ such that these inequalities
are satisfied.

To verify the result of the
proposition for a given sequence of $\ell$ RNM moves applied
to ${\btriangzer}$, we need to compute the flatness of
${\btriangell}$, which is, by construction, equal to
the flatness of $\triang_{k_0+\ell}$ measured in
coordinate frame $\coordframe_2$; see~\eqref{eqn-triangrels}.
We can directly calculate the ratio of the area of $\btriangell$
to the area of ${\btriangzer}$ by using the
number of contractions in the move sequence,
since each contraction multiplies the area by $\half$.
The width of ${\btriangell}$
can be obtained using inequalities and linear polynomials
in $s$, since the width is determined by the largest and smallest
$\xtrans$ coordinates, which are linear polynomials in $s$.
Consequently, the condition that the flatness for each
triangle in the sequence is less than $1.01$ times the original
flatness can be expressed as a Boolean combination of (at most cubic)
polynomial inequalities in $s$, where $s$ is constrained to
satisfy $0 \le s \le 1.00001$.

To determine whether there are allowable
values of $s$ for which a specified sequence of RNM moves is
possible, observe that
a Boolean combination of polynomial inequalities in $s$
will evaluate to ``True'' for $s$ in a certain union
of intervals that can be computed as follows.
We first find the values of $s$ that are solutions
of the polynomial {\sl equations\/} obtained by replacing any
inequalities by equalities.
Then, between each adjacent pair of solutions,
we choose a test value (e.g., the midpoint)
and check whether the associated inequality evaluates to ``True''
on that interval.

The computation time can be cut in half by considering
only sequences that begin with an {\sl inside} contraction,
for the following reason. The outside contraction
point for an original triangle $\triang$ with vertices
$\bp_1$, $\bp_2$, and $\bp_3$ is equal to the
inside contraction
point for a triangle, denoted by $\triang'$, whose worst vertex
$\bp_3$ is the reflection point $\bp\subref$ of $\triang$.
With exact computation, the conditions for an outside
contraction of $\triang$ differ
from those for an inside contraction of $\triang'$
if equality holds in some of the comparisons.
In particular, if $\objfun(\bp_3) > \objfun(\bp\subref) \ge \objfun(\bp_2)$,
then $\triang$ will undergo an outside contraction
and $\triang'$ will undergo an inside contraction; 
but if $\objfun(\bp_3) = \objfun(\bp\subref)$, then both $\triang'$ and
$\triang$ will undergo inside contractions.
Since our inequalities allow for a small error in comparisons,
this difference will not change the result,
and we may assume that the RNM move at
$\triang_{k_0}$ is an inside contraction.

Finally, the definition of the RNM algorithm imposes further
constraints on valid move patterns.
For example, if a reflection occurs, the reflection point
must be strictly better than the second-worst vertex,
so this reflection point
cannot be the worst point in the new triangle.
Such sequences (impossible in the RNM algorithm) would be
permitted by the small error allowed in the inequalities, so
they are explicitly disallowed in the Mathematica code.

Putting all this together, a program can test each sequence
of valid operations that begins with an inside contraction to
determine whether there exists an initial triangle for
which ratio of the flatnesses, measured in $\coordframe_2$,
is less than $1.01$.
The results of this computation show that,
within no more than 14 RNM moves following a contraction,
a triangle is always reached for which the ratio of the
flatnesses, measured in the second coordinate frame $\coordframe_2$,
is at least $1.01$.  We stress that the count of 14 moves includes
a mixture of
reflections and both forms of contraction.  Details of these
move sequences can be found in the appendix.
There we list the $s$-values and the associated
sequences of 14 or fewer RNM moves for which
the ratio of the flatnesses remains less than $1.01$.
\end{proof}

Proposition~\ref{prop-14steps} used $\coordframe_2$,
but its analogue for $\coordframe_1$ follows almost
immediately
with a slightly smaller constant in place of $1.01$.

%
%
\begin{lemma}
\label{lem-14steps}
Under the assumptions of Proposition~\ref{prop-14steps},
there exists $\ell$ with $1 \le \ell \le 14$ such that
\[
     \flatness_{k_0+\ell}^{(1)} >
       1.001\; \flatness_{k_0}^{(1)}.
\]
\end{lemma}

\begin{proof}
The base point of $\coordframe_1$ is the worst point of $\triang_{k_0}$;
the base point of $\coordframe_2$ is the midpoint of the edge
of $\triang_{k_0}$
joining the two vertices whose $\xtrans$ coordinates are
leftmost and rightmost when measured in $\coordframe_1$.
By choosing $k_0$ to be large enough,
the two base points can be made arbitrarily close.
Lemma~\ref{lem-closebase}\eqref{I:closebase flatness}
with $\epsilon = 0.0001$ shows
that for large enough $k_0$, 
the flatnesses of triangles $\triang_{k_0}$ and
$\triang_{k_0+\ell}$ measured in coordinate frames
$\coordframe_1$ and $\coordframe_2$ satisfy
\begin{equation}
\label{eqn-flatcompare}
   0.9999 \; \flatness_{k_0}^{(1)}  \le  \flatness_{k_0}^{(2)}
       \le 1.0001\;\flatness_{k_0}^{(1)}\quad\hbox{and}\quad
      0.9999 \; \flatness^{(2)}_{k_0+\ell}
     \le  \flatness^{(1)}_{k_0+\ell}
    \le 1.0001 \;\flatness^{(2)}_{k_0+\ell}.
\end{equation}
Now, for $\ell$ as in Proposition~\ref{prop-14steps},
\begin{align*}
    \flatness_{k_0+\ell}^{(1)} & \ge  0.9999\; \flatness_{k_0+\ell}^{(2)} \\
                          & > 0.9999 (1.01) \flatness_{k_0}^{(2)} \qquad\textup{(by Proposition~\ref{prop-14steps})} \\
                   & \ge  0.9999(1.01) (0.9999) \flatness_{k_0}^{(1)} \\
                   & >  1.001 \; \flatness_{k_0}^{(1)}.\qedhere
\end{align*}
\end{proof}

%
%
\subsection{Completion of the proof}

The main result of this paper is the following theorem 
(called Theorem~\ref{thm-rnmconv} in Section~\ref{sec-introduction}).

%
%
\begin{theorem}\label{thm-converges}
If the RNM algorithm is applied to 
a function $\objfun \in \calF$,
starting from any nondegenerate triangle, 
then the algorithm converges to the unique minimizer of $\objfun$.
\end{theorem}

\begin{proof}
In this proof, $\flatness_j(\triang_i)$
denotes the flatness of RNM triangle $\triang_i$ measured in
a coordinate frame $\coordframe_{j}$
whose base point is the {\sl worst vertex of triangle $\triang_j$}.

Given a small positive number $\kappa$,
let $k_0$ be sufficiently large
(we will specify how small and how large as we go along).
As mentioned in Section~\ref{sec-bigpic},
the RNM triangle must contract infinitely often,
so we may increase $k_0$ to assume that $\triang_{k_0}$ contracts.
Lemma~\ref{lem-14steps} shows that the flatness measured
in $\coordframe_{k_0}$ increases by a factor of $1.001$ 
in at most $14$ RNM moves;
i.e., there exists $k_1$ with $k_0 < k_1 \le k_0 + 14$
such that 
\begin{equation}
  \label{E:grew by 1.001}
    \flatness_{k_0}(\triang_{k_1}) > 1.001\; \flatness_{k_0}(\triang_{k_0}).
\end{equation}
We now switch coordinate frames on the left hand side:
Lemma~\ref{lem-closebase}\eqref{I:closebase flatness}
and Remark~\ref{R:width greater than height}
show that the flatness of $\triang_{k_1}$ 
in $\coordframe_{k_1}$ is close
to its flatness in $\coordframe_{k_0}$.
In particular, if $k_0$ is sufficiently large, then
\begin{equation}
  \label{E:switch frame 1}
    \flatness_{k_1}(\triang_{k_1}) \ge 0.9999\; \flatness_{k_0}(\triang_{k_1}).
\end{equation}
Let $k_2 \ge k_1$ be the first iteration after (or equal to) $k_1$
such that $\triang_{k_2}$ contracts.
Lemma~\ref{lem-must-contract} shows that
if $k_0$ is sufficiently large,
then from iteration $k_1$ to the beginning of iteration $k_2$,
the distance travelled by the centroid, measured in $\coordframe_{k_1}$,
is less than $\movedist$.
During those iterations, the RNM triangle retains its shape
and hence its flatness, as measured in $\coordframe_{k_1}$;
that is, 
\begin{equation}
  \label{E:equal flatness}
	\flatness_{k_1}(\triang_{k_2}) =  \flatness_{k_1}(\triang_{k_1}).
\end{equation}
If $\kappa$ was small enough,
Lemma~\ref{lem-closebase}\eqref{I:closebase flatness}
and Remark~\ref{R:width greater than height}
again imply 
\begin{equation}
\label{E:switch frame 2}
    \flatness_{k_2}(\triang_{k_2}) \ge 
     0.9999\; \flatness_{k_1}(\triang_{k_2}).
\end{equation}
Combining \eqref{E:grew by 1.001}, \eqref{E:switch frame 1}, 
\eqref{E:equal flatness}, and~\eqref{E:switch frame 2} 
yields 
\[
    \flatness_{k_2}(\triang_{k_2}) >
       (0.9999)^2(1.001) \flatness_{k_0}(\triang_{k_0})
    > 1.0007 \; \flatness_{k_0}(\triang_{k_0}).
\]

If $k_0$ is sufficiently large,
then repeating the process that led from $k_0$ to $k_2$
defines $k_0<k_2<k_4<\cdots$ such that
\[
	\flatness_{k_{2n}}(\triang_{k_{2n}}) >
	(1.0007)^n \flatness_{k_0}(\triang_{k_0})
\]
for all $n$: 
  to know that the same lower bound on $k_0$
works at every stage, we use that 
in Lemma~\ref{lem-closebase}\eqref{I:closebase flatness} 
the number $\delta$ is independent of $\bb_1$, $\bb_2$, and $\triang$.
Now, if $n$ is sufficiently large, then 
\[
	\flatness_{k_{2n}}(\triang_{k_{2n}}) > 10.
\]
But $\triang_{k_{2n}}$ contracts, 
so this contradicts Lemma~\ref{lem-smalldelta}.

Hence the assumption made at the beginning of our long chain
of results, Hypothesis~1,
must be wrong.
In other words, the RNM algorithm {\sl does} converge
to the minimizer of $\objfun$.
\end{proof}

%
\section{Concluding Remarks}
\label{sec-concluding}

\subsection{Why do the McKinnon examples fail?}
\label{sec-mckinfit-old}
For 
general interest, we briefly revisit
the smoothest McKinnon counterexample~\eqref{eqn-mckinsamp}, which
consists of a twice-continuously
differentiable function $f$ and a specific starting triangle for which the
RNM algorithm converges to a nonminimizing point 
(with nonzero gradient).  The Hessian matrix is positive semidefinite
and singular at the limit point, but positive definite everywhere else.
Thus all the assumptions in our
convergence theorem are satisfied except for
positive-definiteness of the Hessian,
which fails at one point.
Hypothesis \ref{hyp1} is valid for this example,
and it is enlightening to examine where the proof by
contradiction fails.

The McKinnon iterates do satisfy several of the
intermediate lemmas in our proof: the RNM triangles not only flatten out
(Lemma~\ref{lem-rattozero}), but they do so more rapidly than the
rate proved in
Lemma~\ref{lem-smallheight}.\footnote{As $k\to\infty$,
the McKinnon triangles satisfy
$\htrans_k\approx \wtrans_k^{\theta}$
for $\theta = \mod{\lambda_2}(1+\mod{\lambda_2})/\lambda_1\approx 3$,
where $\lambda_{1,2} = (1\pm\sqrt{33})/8$.}
However, an essential reduction step,
Lemma~\ref{lem-closebase},
fails to hold for the McKinnon example, as discussed below.

Positive-definiteness of the Hessian plays a crucial role in our proof
by contradiction because it allows us to uniformly approximate the
objective function close to the limit
point $\bplimit$ by its degree-$2$ Taylor polynomial.
Applying a well-defined change of variables, the function $\half x^2 + y$
for a simple triangle can then be taken as a surrogate, and
we can essentially reduce the problem 
to studying the RNM algorithm for 
the objective function $\half x^2 + y$ near the
non-optimal point $(0,0)$.
In the McKinnon example~\eqref{eqn-mckinsamp}, however,
the objective function near the limit point $(0,0)$
cannot be (uniformly) well approximated by $\half x^2 + y$, 
even after a change of variable. 
Although the Hessian of the McKinnon function $f$ remains positive definite
at base points in $\triang_k$ as $k\to\infty$,
it becomes increasingly close to singular, in such a way that 
ever-smaller changes in the base point will eventually not 
satisfy the closeness conditions of Lemma~\ref{lem-closebase}.
In fact, the actual shape of the McKinnon objective function 
allows a sequence of RNM moves that are forbidden for $\half x^2 + y$
near the non-optimal point $(0,0).$
namely an infinite sequence of inside contractions with the best vertex
never replaced. 
In dynamical terms, the
McKinnon objective function allows symbolic dynamics  forbidden for
$\half x^2 + y$ near $(0,0)$, 
and these symbolic dynamics  evade the contradiction in our argument.

\subsection{An instance of RNM convergence}
Most of this paper has been devoted to analysis of situations
that we subsequently show cannot occur; this is the nature of
arguments by contradiction.  For contrast, we present one
example where the RNM algorithm will converge, as we have proved,
on the strictly convex quadratic function
$$
  f(x,y) = 2x^2 + 3y^2 + xy - 3x + 5y,
$$
whose minimizer is $x^* = (1, -1)^T$.  Using starting vertices
$(0, 0.5)\T$, $(0.25, -0.75)\T$, and
$(-0.8, 0)\T$, after 20 RNM iterations the best vertex 
is $(0.997986, -1.00128)\T$, and the RNM triangles are obviously
converging to the solution. The first nine iterations are
depicted in Figure~\ref{fig-rnmquad}.

%
%
\begin{figure}[htb]
\begin{center}
{\includegraphics[width=4in,height=2.25in]{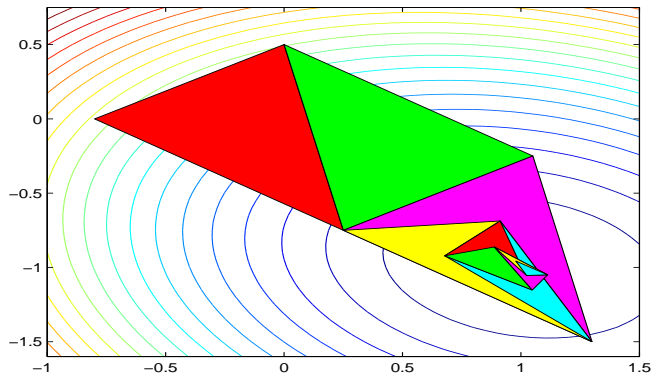}}
\caption{\small Convergence of the RNM algorithm
on a strictly convex quadratic function.}
\label{fig-rnmquad}
\end{center}
\end{figure}

\subsection{Significance of the results in this paper}

This paper began by noting that very little is known about the
theoretical properties of the original Nelder--Mead method,
despite 45 years of practice.  It is fair to say that proving convergence
for an RNM algorithm in two dimensions on a
restricted class of functions adds only a little more to
this knowledge.  This contribution seems of interest, however,
because of the lack of other results despite determined
efforts, and the introduction of dynamical systems  methods 
to the analysis.

Our analysis applies only to  a simplified (``small step") version of the 
original Nelder--Mead method which excludes expansion steps. 
 We have observed that in thousands of computational
experiments with functions defined in
$\real^n$ ($n\ge 2$) in which the Nelder--Mead method converges to a
minimizer, expansion steps are almost never taken in the
neighborhood of the optimum.  Expansion steps are typically taken
early on, forming part of the ``adaptation to the local contours''
that constituted the motivation for Nelder and Mead when they
originally conceived the
algorithm \cite{NM}.  Thus the RNM algorithm appears to represent,
to a large extent, the behavior of the original method near the
solution. In this direction, it would be valuable if these empirical observations 
could be rigorously justified under a
well-defined set of conditions.  The
observed good performance of the Nelder--Mead method
on many real-world problems remains a puzzle.

This paper applies dynamical systems methods to the
analysis of the RNM algorithm.
The use of such ideas in the proofs,
particularly that of a (rescaled) local coordinate frame in 
Section~\ref{sec-rnminequalities}, 
may also be useful in other contexts where it is valuable to
connect the geometry of a simplex with the contours of the objective
function. 
The evolving geometric figures  of the algorithm remain  one of
the intuitive appeals of the original Nelder--Mead method,
leading to the nickname of ``amoeba method''  \cite{Recipes}.
There may well be other applications, but 
the latest direct search methods tend to exhibit a less
clear connection with geometry.

Finally, our analysis for the RNM algorithm relies in part on
the fact that the volume of the RNM simplex is non-increasing
at every iteration, thereby avoiding the difficulties
associated with expansion steps.  Consequently, McKinnon's question
remains open: does the original Nelder--Mead algorithm,
including expansion steps, always converge for the function
$x^2+y^2$, or more generally for a class of functions like those
treated in Theorem~\ref{thm-converges}?  We hope that
further development of the dynamical systems approach could lead to
progress on this question.

%

\section*{Appendix: Computation for Proposition~\ref{prop-14steps}.}

This appendix provides details of the symbolic computation
performed to prove Proposition~\ref{prop-14steps}.
We regard the coding of moves as a form of {\sl symbolic dynamics\/}
for the RNM iteration. Moves are represented as follows:
{\tt 1}, {\tt 2}, and {\tt 3} denote reflections
with, respectively, vertex ${\bA}$, ${\bB}$, or ${\bC}$ 
of~\eqref{eqn-btriangzerverts}
taken as the worst vertex, i.e.\ replaced during the move.
Similarly, {\tt 4}, {\tt 5}, and {\tt 6} denote inside contractions,
and {\tt 7}, {\tt 8}, {\tt 9}  denote outside contractions with
worst vertex ${\bA}, {\bB}, {\bC},$ respectively.

We describe a sequence of move numbers as {\sl possible}\ 
for a given $s \in [0,1.00001]$ 
if there exist $t,u \in [-40.0005,40.0005]$
such that 
for the triangle~\eqref{eqn-btriangzerverts} described by $(s,t,u)$,
\begin{enumerate}[\upshape (i)]
\item 
the variables $s,t,u$ satisfy the inequality
implied by~\eqref{eqn-fimplicationone} for each RNM move,
\item
the flatness after each step is less than or equal to
$1.01$ times the original flatness, and
\item 
no reflection undoes an immediately preceding reflection.
\end{enumerate}

\begin{remark}
Because~\eqref{eqn-fimplicationone} involves a relaxation of $10^{-6}$,
a sequence characterized as ``possible'' using the
first two properties listed above
could be impossible for the RNM algorithm
{\sl in exact arithmetic}.
This is why the third condition
explicitly prohibits sequences in
which a reflection undoes the previous move, something that can never
happen in the RNM algorithm.
\end{remark}

In the proof of Proposition~\ref{prop-14steps},
we described a symbolic algorithm for computing all possible sequences
beginning with an inside contraction.
The Mathematica output below lists all these sequences.

\begin{verbatim}
{5} possible for s in {{0.999999, 1.00001}}
{5, 6} possible for s in {{0.999999, 1.00001}}
{6} possible for s in {{0.582145, 1.}}
{6, 2} possible for s in {{0.582145, 0.737035}}
{6, 2, 1} possible for s in {{0.582145, 0.695708}}
{6, 2, 1, 3} possible for s in {{0.582145, 0.654949}}
{6, 2, 1, 3, 2} possible for s in {{0.582145, 0.654949}}
{6, 2, 1, 3, 6} possible for s in {{0.582145, 0.654949}}
{6, 2, 1, 3, 6, 2} possible for s in {{0.616769, 0.654949}}
{6, 2, 1, 3, 6, 2, 5} possible for s in {{0.616769, 0.64706}}
{6, 2, 1, 3, 6, 8} possible for s in {{0.582145, 0.64706}}
{6, 2, 1, 3, 6, 8, 4} possible for s in {{0.582145, 0.623495}}
{6, 2, 1, 3, 9} possible for s in {{0.582145, 0.644579}}
{6, 2, 1, 6} possible for s in {{0.582145, 0.695708}}
{6, 2, 1, 9} possible for s in {{0.582145, 0.673138}}
{6, 2, 1, 9, 2} possible for s in {{0.616769, 0.673138}}
{6, 2, 1, 9, 2, 5} possible for s in {{0.616769, 0.64706}}
{6, 2, 1, 9, 8} possible for s in {{0.582145, 0.64706}}
{6, 2, 1, 9, 8, 4} possible for s in {{0.582145, 0.623495}}
{6, 2, 5} possible for s in {{0.582145, 0.737035}}
{6, 2, 5, 4} possible for s in {{0.582145, 0.695708}}
{6, 2, 5, 7} possible for s in {{0.582145, 0.681931}}
{6, 2, 5, 7, 6} possible for s in {{0.582145, 0.635866}}
{6, 2, 5, 7, 9} possible for s in {{0.582145, 0.681931}}
{6, 2, 5, 7, 9, 5} possible for s in {{0.582145, 0.679967}}
{6, 2, 5, 7, 9, 8} possible for s in {{0.582145, 0.663254}}
{6, 2, 5, 7, 9, 8, 4} possible for s in {{0.582145, 0.646912}}
{6, 2, 5, 7, 9, 8, 7} possible for s in {{0.582145, 0.663254}}
{6, 2, 5, 7, 9, 8, 7, 6} possible for s in {{0.582145, 0.663254}}
{6, 2, 5, 7, 9, 8, 7, 6, 5} possible for s in {{0.589537, 0.663254}}
{6, 2, 5, 7, 9, 8, 7, 6, 5, 1} possible for s in {{0.589537, 0.635373}}
{6, 2, 5, 7, 9, 8, 7, 9} possible for s in {{0.582145, 0.65445}}
{6, 2, 5, 7, 9, 8, 7, 9, 5} possible for s in {{0.582145, 0.651784}}
{6, 2, 5, 7, 9, 8, 7, 9, 5, 4} possible for s in {{0.582145, 0.651784}}
{6, 2, 5, 7, 9, 8, 7, 9, 5, 4, 3} possible for s in {{0.582145, 0.651784}}
{6, 2, 5, 7, 9, 8, 7, 9, 8} possible for s in {{0.597869, 0.65445}}
{6, 2, 5, 7, 9, 8, 7, 9, 8, 4} possible for s in {{0.597869, 0.65445}}
{6, 2, 5, 7, 9, 8, 7, 9, 8, 4, 6} possible for s in {{0.597869, 0.65445}}
{6, 2, 5, 7, 9, 8, 7, 9, 8, 4, 6, 2} possible for s in {{0.597869, 0.654004}}
{6, 2, 5, 7, 9, 8, 7, 9, 8, 4, 6, 2, 5} possible for s in {{0.64094, 0.654004}}
{6, 2, 5, 7, 9, 8, 7, 9, 8, 4, 6, 8} possible for s in {{0.64094, 0.65445}}
{6, 2, 8} possible for s in {{0.582145, 0.614711}}
{6, 5} possible for s in {{0.582145, 1.}}
{6, 8} possible for s in {{0.582145, 0.853944}}
{6, 8, 4} possible for s in {{0.582145, 0.810502}}
{6, 8, 7} possible for s in {{0.582145, 0.853944}}
{6, 8, 7, 6} possible for s in {{0.582145, 0.853944}}
{6, 8, 7, 9} possible for s in {{0.582145, 0.818183}}
{6, 8, 7, 9, 5} possible for s in {{0.582145, 0.811611}}
{6, 8, 7, 9, 8} possible for s in {{0.582145, 0.818183}}
{6, 8, 7, 9, 8, 4} possible for s in {{0.582145, 0.818183}}
{6, 8, 7, 9, 8, 4, 6} possible for s in {{0.763168, 0.818183}}
{6, 8, 7, 9, 8, 4, 6, 2} possible for s in {{0.763168, 0.817831}}
{6, 8, 7, 9, 8, 7} possible for s in {{0.582145, 0.777853}}
{6, 8, 7, 9, 8, 7, 6} possible for s in {{0.582145, 0.777853}}
{6, 8, 7, 9, 8, 7, 6, 5} possible for s in {{0.589537, 0.777853}}
{6, 8, 7, 9, 8, 7, 6, 5, 1} possible for s in {{0.589537, 0.777853}}
{6, 8, 7, 9, 8, 7, 9} possible for s in {{0.582145, 0.751661}}
{6, 8, 7, 9, 8, 7, 9, 5} possible for s in {{0.582145, 0.751661}}
{6, 8, 7, 9, 8, 7, 9, 5, 4} possible for s in {{0.582145, 0.751661}}
{6, 8, 7, 9, 8, 7, 9, 5, 4, 3} possible for s in {{0.582145, 0.751661}}
{6, 8, 7, 9, 8, 7, 9, 8} possible for s in {{0.597869, 0.694824}}
{6, 8, 7, 9, 8, 7, 9, 8, 4} possible for s in {{0.597869, 0.694824}}
{6, 8, 7, 9, 8, 7, 9, 8, 4, 6} possible for s in {{0.597869, 0.694824}}
{6, 8, 7, 9, 8, 7, 9, 8, 4, 6, 2} possible for s in {{0.597869, 0.694824}}
{6, 8, 7, 9, 8, 7, 9, 8, 4, 6, 2, 5} possible for s in {{0.64094, 0.663616}}
{6, 8, 7, 9, 8, 7, 9, 8, 4, 6, 8} possible for s in {{0.64094, 0.663616}}
\end{verbatim}
All we need from this computation is that there is no possible sequence 
of $14$ steps or more.
In other words, following an inside contraction,
the flatness will be greater than $1.01$ times the original flatness
after no more than $14$ steps (including the initial contraction).

\subsection*{Remarks about the list of possible sequences}

The remarks in this section are not needed for the proof,
but they may give further insight into the behavior of the RNM algorithm
as well as clear up some potential ambiguity about the computer output above.

\begin{itemize}
\item
That the sequence $\{ {\tt 4} \}$ is not possible 
(i.e., that an inside contraction with ${\bA}_0$ as worst vertex
immediately increases the flatness by at least a factor of $1.01$)
was shown already near the beginning of 
the proof of Proposition~\ref{prop-14steps}.
\item 
The bound $40.0005$ on $|t|$ and $|u|$ need not be fed into the program,
because the program automatically calculates stronger inequalities
that are necessary for a contraction to occur.
\item 
Move sequences that do not appear in the list may still occur
in actual runs of the RNM algorithm,
but then the flatness must grow by more than a factor of $1.01$.
Similarly, a move sequence appearing in the list may occur while
running the RNM algorithm
even if $s$ lies outside the given interval.
For example, 
one can show that there exist triangles with $0 \le s < 0.582145$
on which the RNM algorithm takes move $\{{\tt 6}\}$.
\item 
One cannot predict from the list {\sl which} step causes the
flatness to grow beyond the factor of $1.01$.
For example, using our definition
the sequence $\{{\tt {6,2, 1, 3, 2}}\}$ is possible
(for a certain range of $s$),
but the extended sequence
$\{ \tt{6, 2, 1, 3, 2, 1}\}$ is not.
This should not be taken to mean that
the last reflection
$\{\tt{1}\}$ caused the increase in flatness, since
reflections do not change the flatness
(measured in the same coordinate frame).
Rather, there may
exist a triangle in the given range
that for the objective function $f(\lambda,\mu)= \frac{1}{2} \lambda^2+\mu$ 
will take the sequence of steps $\{ \tt{6, 2, 1, 3, 2, 1}\}$. 
What must be the case, however, is that
for any such triangle the initial inside contraction $\{ \tt{6} \}$
will have already increased the invariant by a factor at least $1.01$.
\item
One cannot deduce that in every run of the RNM algorithm,
every sufficiently advanced sequence of 14 steps
involves a contraction.
Experiments show that,
when omitting any test for flatness, 
a sequence beginning with $\{{\tt 6}\}$ can
legitimately be followed by a very large number of reflect steps during
which the flatness does not change.
Thus we truly needed Lemma~\ref{lem-must-contract} in addition to
Proposition~\ref{prop-14steps} to complete our proof.
\item
The entire computation took about 11 minutes on an Intel Xeon 3.0~GHz processor.
\end{itemize}

%



\begin{thebibliography}{99}


\bibitem{Audet04}
C.\ Audet (2004).
Convergence results for pattern search algorithms
are tight, {\sl Optimization and Engineering\/} 5, 101--122.

\bibitem{AD03}
C.\ Audet and J.\ E.\ Dennis, Jr. (2003).
Analysis of generalized pattern searches,
{\sl SIAM Journal on Optimization\/} 13, 889-903.

\bibitem{AD06}
C.\ Audet and J.\ E.\ Dennis, Jr. (2006).
Mesh adaptive direct search algorithms for constrained
optimization, {\sl SIAM Journal on Optimization\/} 17,
188--217.


\bibitem{Bertnotes}
D.\ Bertsekas (2003).
{\sl Convex Analysis and Optimization},
Athena Scientific.


\bibitem{CSV09}
A.\ R.\ Conn, K.\ Scheinberg, and L.\ N.\ Vicente (2009).
{\sl Introduction to Derivative-Free Optimization},
Society for Industrial and Applied Mathematics,
Philadelphia, Pennsylvania.

\bibitem{CoopePrice01}
I.\ D.\ Coope and C.\ J.\ Price (2001).
On the convergence of grid-based methods for unconstrained
optimization, {\sl SIAM Journal on Optimization\/} 11,
859--869.

\bibitem{gsl}
GNU Scientific Library (2011). NLopt,
Free Software Foundation, Boston, Massachusetts.\\
{\tt ab-initio.mit.edu/wiki/index.php/NLopt}\_{\tt Algorithms}

\bibitem{Gurson00}
A.\ P.\ Gurson (2000).
``Simplex search behavior in nonlinear optimization'',
Bachelor's honors thesis, Computer Science Department,
College of William and Mary, Williamsburg, Virginia.\\
{\tt www.cs.wm.edu/}$\sim${\tt va/CS495}

\bibitem{HanNeumann06}
L.\ Han and M.\ Neumann (2006).
Effect of dimensionality on the Nelder--Mead simplex method,
{\sl Optimization Methods and Software\/} 21, 1--16.

\bibitem{HSW88}
D.\ Hensley, P.\ Smith, and D.\ Woods (1988).
Simplex distortions in Nelder--Mead reflections,
IMSL Technical Report Series No. 8801,
IMSL, Inc., Houston, Texas.

\bibitem{Kel99}
C.\ T.\ Kelley (1999).
Detection and remediation of stagnation in the Nelder--Mead
algorithm using a sufficient decrease condition,
{\sl SIAM Journal on Optimization\/} 10, 43--55.

\bibitem{Kelley-book}
C.\ T.\ Kelley (1999).
{\sl Iterative Methods for Optimization\/},
Society for Industrial and Applied Mathematics, Philadelphia, Pennsylvania.

\bibitem{Klein}
F.\ Klein (1939).
{\sl Elementary Mathematics from an Advanced Standpoint: Geometry},
 Dover Publications: New York. (Reprint of Volume II of English translation
 of F. Klein, {\sl Elementarmathematik vom H\"{o}heren Standpunkte aus,}
 J. Springer, Berlin 1924--1928.)

\bibitem{KLT-SIREV}
T.\ G.\ Kolda, R.\ M.\ Lewis, and V.\ Torczon (2003).
Optimization by direct search: new perspectives
on some classical and modern methods,
{\sl SIAM Review\/} 45, 385--482.

\bibitem{LRWW98}
J.\ C.\ Lagarias, J.\ A.\ Reeds, M.\ H.\ Wright and P.\ E.\ Wright (1998).
Convergence properties of the Nelder--Mead simplex
algorithm in low dimensions,
{\sl SIAM Journal on Optimization} 9, 112--147.


\bibitem{LTT00}
R.\ M.\ Lewis, V.\ Torczon, and M.\ W.\ Trosset (2001).
Direct search methods: then and now,
in {\sl Numerical Analysis 2000}, Volume 4, 191--207, Elsevier,
New York.

\bibitem{matlab}
{{\sl MATLAB\/}\texttrademark} {\sl User's Guide} (2010). 
R2010b Documentation, The Mathworks, Inc.,
Natick, Massachusetts.\\
{\tt www.mathworks.com/help/techdoc/ref/fminsearch.html}

\bibitem{McKinnon98}
K.\ I.\ M. McKinnon (1998).
Convergence of the Nelder--Mead simplex method
to a non-stationary point,
{\sl SIAM Journal on Optimization} 9, 148--158.

\bibitem{NT02}
L.\ J.\ Nazareth and P.\ Tseng (2002).
Gilding the lily: A variant of the Nelder--Mead algorithm based
on golden section search,
{\sl Computational Optimization and Applications } 22, 133--144.


\bibitem{NM}
J.\ A.\ Nelder and R.\ Mead (1965).
A simplex method for function minimization,
{\sl Computer Journal\/} 7, 308--313.

\bibitem{OR}
J.\ M.\ Ortega and W.\ C.\ Rheinboldt (1970).
{\sl Iterative Solution of Nonlinear Equations in Several Variables},
Academic Press, New York.

\bibitem{PowellActa}
M.\ J.\ D.\ Powell (1998).
Direct search algorithms for optimization calculations,
{\sl Acta Numerica\/} 7, 287--336.

\bibitem{Recipes}
W.\ H.\ Press, B.\ P.\ Flannery, S.\ A.\ Teukolsky, and
W.\ T.\ Vettering (1992). {\sl Numerical Recipes in Fortran: the Art of
Scientific Computing} (second ed.),
Cambridge University Press, Cambridge, UK.

\bibitem{Priceetal02}
C.\ J.\ Price, I.\ D.\ Coope, and D.\ Byatt (2002).
A convergent variant of the Nelder--Mead algorithm,
{\sl Journal of Optimization Theory and Applications\/} 113,
5--19.

\bibitem{Rykov83}
A.\ S.\ Rykov (1983).
Simplex algorithms for unconstrained optimization,
{\sl Problems of Control and Information Theory\/} 12, 195--208.

\bibitem{Schrijver}
A.\ Schrijver (1987). {\sl Theory of Linear and Integer Programming},
John Wiley and Sons, New York.

\bibitem{StuHumph}
A.\ M.\ Stuart and A.\ R.\ Humphries (1996). {\sl Dynamical Systems
and Numerical Analysis}, Cambridge University Press, Cambridge, UK.

\bibitem{Torczon-pstheory}
V.~Torczon (1997).
On the convergence of pattern search algorithms,
{\sl SIAM Journal on Optimization\/} 7, 1--25.

\bibitem{Tse99}
P. ~ Tseng (1999).
Fortified-descent simplicial search method: A general approach,
{\sl SIAM Journal on Optimization}, 10, No. 1, 269--288.

\bibitem{WoodsPhD}
D.\ J.\ Woods (1985).
{\sl An Interactive Approach for Solving Multi-Objective
Optimization Problems}, PhD thesis, Technical Report 85-5,
Department of Computational and Applied Mathematics, Rice University,
Houston, Texas.

\bibitem{Dundee}
M.~H. Wright (1996).
Direct search methods: once scorned, now respectable.
in {\sl Numerical Analysis 1995: Proceedings of the 1995
Dundee Biennial Conference in Numerical Analysis},
D.\ F.\ Griffiths and G.\ A.\ Watson (eds.),
191--208,
Addison Wesley Longman, Harlow, UK.

\end{thebibliography}

\end{document}